\documentclass{amsart}

\usepackage{mathrsfs}
\usepackage{amscd}
\usepackage{amsmath}
\usepackage{amssymb}
\usepackage{amsthm}
\usepackage{epsf}
\usepackage{latexsym}
\usepackage{verbatim}
\usepackage[all, cmtip]{xy}
\usepackage{tikz}
\usetikzlibrary{positioning}
\usetikzlibrary{matrix}
\usepackage{float}
\usepackage{hyperref}
\usepackage{comment}
\usepackage{enumitem}

\tikzstyle{bsq}=[rectangle, draw, thick, minimum width=.5cm, minimum height=.5cm]
\tikzstyle{bver}=[rectangle, draw, thick, minimum width=1cm, minimum height=2cm]
\tikzstyle{bhor}=[rectangle, draw, thick, minimum width=2cm, minimum height=1cm]

\usepackage[left=3.2cm, right=3.2cm]{geometry}

\newtheorem{theorem}{Theorem}[section]
\newtheorem{definition}[theorem]{Definition}
\newtheorem{lemma}[theorem]{Lemma}

\newtheorem{corollary}[theorem]{Corollary}
\newtheorem{proposition}[theorem]{Proposition}
\newtheorem{question}[theorem]{Question}
\newtheorem{varexample}[theorem]{Example}

\theoremstyle{definition}
\newtheorem{remark}[theorem]{Remark}

\newcommand{\Spec}{\mathrm{Spec}\,}

\def\supp{\mathrm{supp}}

\newcommand{\cL}{\mathcal{L}}

\newcommand{\ord}{\operatorname{ord}}
\newcommand{\Trop}{\operatorname{Trop}}
\newcommand{\trop}{\operatorname{trop}}
\newcommand{\ddiv}{\operatorname{div}}
\newcommand{\Div}{\operatorname{Div}}
\newcommand{\PL}{\operatorname{PL}}

\newcommand{\val}{\operatorname{val}}
\newcommand{\Pic}{\operatorname{Pic}}
\newcommand{\Sym}{\operatorname{Sym}}

\newcommand{\an}{\mathrm{an}}

\newenvironment{example}{\begin{varexample}
\begin{normalfont}}{\end{normalfont}
\end{varexample}}

\begin{document}
\title{Tropical linear series and tropical independence}
\author[D. Jensen]{David Jensen}
\address{David Jensen: Department of Mathematics,  University of Kentucky \hfill \newline\texttt{}
\indent 733 Patterson Office Tower,
Lexington, KY 40506--0027, USA}
\email{{\tt dave.jensen@uky.edu}}
\author[S. Payne]{Sam Payne}
\address{Sam Payne: Department of Mathematics,  University of Texas at Austin \hfill \newline\texttt{}
\indent 2515 Speedway, PMA 8.100,
Austin, TX 78712, USA}
\email{{\tt sampayne@utexas.edu}}
\date{}
\bibliographystyle{alpha}

\maketitle

\begin{abstract}
We propose a definition of tropical linear series that isolates some of the essential combinatorial properties of tropicalizations of not-necessarily-complete linear series on algebraic curves. The definition combines the Baker-Norine notion of rank with the notion of tropical independence and has the property that the restriction of a tropical linear series of rank $r$ to a connected subgraph is a tropical linear series of rank $r$.  We show that tropical linear series of rank 1 are finitely generated as tropical modules and state a number of open problems related to algebraic, combinatorial, and topological properties of higher rank tropical linear series.
\end{abstract}

\section{Introduction}
\label{Sec:Intro}

Fifteen years ago, Baker and Norine introduced the rank of a divisor on a graph, in close analogy with the rank of a divisor on an algebraic curve, and proved the surprising and beautiful fact that it satisfies a precise analogue of the Riemann-Roch Theorem \cite{BakerNorine07}.  This breakthrough inspired several new directions of research.  Baker proved the Specialization Lemma \cite{Baker08} relating ranks of divisors on curves to those on graphs via semistable degenerations, and outlined a program for relating divisor theory on graphs to the celebrated results of Brill-Noether theory on algebraic curves.  All of these results extend naturally to tropical curves (metric graphs) \cite{GathmannKerber08, MikhalkinZharkov08}, with specialization given by retraction to the skeleton of the Berkovich analytifications of curves over valued fields.  Much of Baker's program has now been carried out, including tropical proofs of the Brill-Noether and Gieseker-Petri Theorems \cite{tropicalBN, tropicalGP}. 

The tropical method envisioned by Baker, using combinatorics of divisors on graphs to prove results about divisors on general algebraic curves, is particularly powerful when used in combination with other tools.  For instance, in combination with deformation theory and log geometry, it led to a proof of Pflueger's conjecture \cite{Pflueger17b} for the dimension of the space of linear series of given degree and rank on a general curve of fixed gonality \cite{JensenRanganathan21}, opening the door to a range of subsequent developments in Hurwitz-Brill-Noether theory \cite{CPJ20, Larson21, LLV}.  

The present paper is an attempt to isolate and formalize the new combinatorial structures that arise in further applications of tropical geometry to linear series on algebraic curves and geometry of moduli spaces, as in \cite{M23, M13}.

\subsection{Algebraic and tropical complete linear series} Given a divisor $D_C$ on a curve $C$, the \emph{complete linear series} $\mathcal{L}(D_C)$ is the set of nonzero rational functions on $C$ with poles bounded by $D_C$, together with 0.  
If $C$ is defined over a valued field whose skeleton is a metric graph $\Gamma$, then the tropicalizations of nonzero rational functions on $C$ are precisely the continuous PL functions with integer slopes on $\Gamma$ \cite[Theorem~1.1]{BakerRabinoff15}. The bend locus of a PL function is the analogue of the zeros and poles of a rational function, with the analogue of order of vanishing given by the sum of the incoming slopes at a point. These notions are compatible with specialization, i.e., the sum of the incoming slopes of $\trop(f)$ at a point is the sum of the multiplicities of the zeros and poles of $f$ that specialize to that point.

Thus one naturally defines the \emph{complete linear series} of a divisor $D$ on $\Gamma$ to be
\[
R(D) := \{ \varphi \in \PL (\Gamma) \mbox{ } \vert \mbox{ } \ddiv (\varphi) + D \geq 0 \} ,
\]
and if $D = \Trop(D_C)$ then $\trop(\mathcal{L}(D_C)) \subseteq R(D)$.  Note that $R(D)$ is a tropical module, meaning that it is closed under scalar addition and pointwise minimum, i.e.,  the tropical analogue of a vector space.  It also comes with a polyhedral structure, and it is both finitely generated as a tropical module, and finite dimensional as a polyhedral complex \cite{HMY12}. However, neither the minimal number of generators nor the polyhedral dimension typically agrees with the rank of $D$. Most graphs $\Gamma$ have divisors $D$ such that the dimension of $R(D)$ is larger than the rank of $D$, and in such cases $R(D)$ is not the tropicalization of any linear series.   See Examples~\ref{Ex:Lollipop} and~\ref{Ex:Barbell}. 

More recent applications of tropical geometry to linear series on algebraic curves \cite{M23, M13} require combinatorial objects that capture more information about tropicalizations of algebraic linear series, including proper linear subspaces of $\mathcal{L}(D_C)$, which are called \emph{incomplete linear series}.  The properties of tropicalizations of possibly incomplete linear series that are most important to these applications depend on the notion of tropical independence that was introduced in \cite{tropicalGP}.

\subsection{Tropical dependence and independence} A tropical linear combination of a collection of functions $\{\varphi_0, \ldots, \varphi_r\} \subset \PL(\Gamma)$ is an expression
\begin{equation}\label{eq:vartheta}
\vartheta = \min \{ \varphi_0 + a_0, \ldots, \varphi_r + a_r \},
\end{equation}
where $a_0, \ldots, a_r$ are real numbers.  Note that we consider the tuple $(a_0, \ldots, a_r)$ to be part of the data of a tropical linear combination, even though different tuples may give rise to the same pointwise minimum $\vartheta$.  

\begin{definition}
The set $\{\varphi_0, \ldots, \varphi_r\} \subset \PL(\Gamma)$ is \emph{dependent} if there is a tropical linear combination $\vartheta$ such that the minimum in \eqref{eq:vartheta} is achieved at least twice at every point in $\Gamma$.
\end{definition}

Such a tropical linear combination is a \emph{dependence}.  In other words, $\vartheta$ is a dependence if 
\[
\vartheta = \min_{j \neq i} \{\varphi_j + a_j \} \mbox{ \ for each \ } 0 \leq i \leq r.
\]
Equivalently, $\vartheta$ is a dependence if no term in \eqref{eq:vartheta} achieves the minimum uniquely at any point of $\Gamma$.  If no such dependence exists, then $\{ \varphi_0, \ldots, \varphi_r\}$ is \emph{independent}.

\begin{lemma}[{\cite[Theorem~1.6]{M23}}]
Suppose there is some $\vartheta = \min \{ \varphi_0 + a_0, \ldots, \varphi_r + a_r \}$ such that 
\begin{equation} \label{eq:independence}
\vartheta \neq \min_{j \neq i} \{\varphi_j + a_j\}, \mbox{ \  for each \ } 0 \leq i \leq r.
\end{equation}
Then $\{ \varphi_0, \ldots, \varphi_r\}$ is independent.
\end{lemma}

\noindent We say that a tropical linear combination $\vartheta$ satisfying \eqref{eq:independence} is a \emph{certificate of independence}.  In other words, $\vartheta$ is a certificate of independence if each of its tropical summands $\varphi_j + a_j$ achieves the minimum uniquely in \eqref{eq:vartheta} at some point of $\Gamma$. Constructing a certificate of independence can be an efficient way to prove that a given set $\{\varphi_0, \ldots, \varphi_r\} \subset \PL(\Gamma)$ is independent, and the following theorem guarantees that such certificates exist for any independent set.

\begin{theorem}[{\cite[Theorem~1.6]{M23}}]
\label{Thm:Ind}
Suppose $\{\varphi_0, \ldots, \varphi_r\} \subset \PL(\Gamma)$ is independent.  Then there are real numbers $a_0, \ldots, a_r$ such that $\vartheta = \min\{ \varphi_0 + a_0, \ldots, \varphi_r + a_r \}$ is a certificate of independence.  
 \end{theorem}

\noindent Theorem~\ref{Thm:Ind} is a consequence of the Knaster-Kuratowski-Mazurkiewicz lemma, a set-covering variant of the Brouwer fixed-point theorem.  The fact that independence is equivalent to the existence of such a special tropical linear combination, i.e. a certificate of independence, has no analogue in classical linear algebra.

\medskip

Applications of tropical independence are based on the following version of specialization.  If $f_0, \ldots, f_r$ are rational functions on a curve over a valued field with skeleton $\Gamma$, and if $c_0 f_0 + \cdots + c_r f_r$ is a linear dependence, then $\vartheta = \min\{ \trop(f_0) + \val (c_0), \ldots, \trop(f_r) + \val(c_r) \}$ is a tropical dependence.  Hence the independence of $\{\trop(f_0), \ldots, \trop(f_r)\}$ is a sufficient condition for linear independence of  $\{f_0, \ldots, f_r\}$.  Many fundamental questions about the geometry of algebraic curves can be expressed in terms of ranks of linear maps between linear series, and constructing a certificate of independence for the tropicalizations of some collection of rational functions in the image gives a lower bound on these ranks.  Note that the image is a possibly incomplete linear series.

\subsection{Tropicalizations of possibly incomplete linear series}

Let $D$ be the tropicalization of a divisor $D_C$ and let $\Sigma$ be the tropicalization of a linear series $V \subseteq \mathcal{L}(D_C)$ of rank $r$, which we do not assume to be complete.  Then $\Sigma$ has the following properties (Proposition~\ref{Prop:Tropicalization}):

\begin{itemize}
\item  (rank)  for every effective divisor $E$ of degree $r$, there is a $\varphi \in \Sigma$ such that $\ddiv (\varphi ) + D \geq  E$;
\item  (dependence)  every set of $r+2$ functions in $\Sigma$ is tropically dependent;
\end{itemize}
A consequence of these first two properties is that, for each tangent vector $\zeta$, there are exactly $r+1$ slopes $s_{\zeta} (\varphi)$, for $\varphi \in \Sigma$ (see Lemma~\ref{Lem:Slopes}).  We denote these $r+1$ slopes by
\[
s_\zeta[0] < \cdots < s_\zeta[r] .
\]
The collections of slope vectors $s_\zeta(\Sigma)$ and their combinatorial properties were first studied systematically by Amini, in his work on equidistribution of Weierstrass points \cite[Section~2.1]{Amini14}; they are naturally identified with the vanishing sequences at nodes in the Amini-Baker theory of linear series on metrized complexes of curves \cite{AminiBaker15}.

The set $\Sigma$ also satisfies:
\begin{itemize}
\item  (subseries)  for each tangent vector $\zeta$, the subset $\{ \varphi \in \Sigma \mbox{ } \vert \mbox{ } s_{\zeta} (\varphi) \leq s_{\zeta} [i] \}$
contains the tropicalization of a linear series of rank $i$, for $0 \leq i < r$.
\end{itemize}

These are the essential properties of tropicalizations of linear series that are used for the applications in \cite{M23, M13}.  Only the first is shared by all complete linear series $R(D)$ of rank $r$.  We therefore propose the following definition.  
\begin{definition}
\label{Def:TLS}
Let $D$ be a divisor on a metric graph $\Gamma$.  A \emph{tropical linear series of rank $r$} is a tropical submodule $\Sigma \subseteq R(D)$ satisfying the three properties listed above, i.e.,
\begin{enumerate}[series=conditions]
\item  for every effective divisor $E$ of degree $r$, there is some $\varphi \in \Sigma$ such that $\ddiv (\varphi ) + D \geq  E$;
\item   every set of $r+2$ functions in $\Sigma$ is tropically dependent;
\item  for each tangent vector $\zeta$, the subset $\{ \varphi \in \Sigma \mbox{ } \vert \mbox{ } s_{\zeta} (\varphi) \leq s_{\zeta} [i] \}$
contains a tropical linear series of rank $i$, for $0 \leq i < r$.
\end{enumerate}
\end{definition}

\noindent Note that this definition is recursive in $r$.  When $r = 0$, the last condition is vacuous; a tropical linear series of rank $0$ is a principal submodule of $\PL(\Gamma)$, i.e., $\{ \varphi + a \mbox{ } \vert \mbox{ } a \in \mathbb{R} \}$, for some $\varphi \in \PL(\Gamma)$.

Tropicalizations of linear series satisfy additional geometric and combinatorial conditions, see Proposition~\ref{Prop:Tropicalization} below.  We do not know whether these properties follow from Definition~\ref{Def:TLS}.  Instead, we define a strong tropical linear series to include these additional conditions.  The less restrictive definition above suffices for the arguments in \cite{M23, M13}.  Let $\vert \Sigma \vert$ be the tropical projectivization $\vert \Sigma \vert := \Sigma /\mathbb{R}$, where $\mathbb{R}$ acts by scalar addition.  Note that $\vert \Sigma \vert$ parametrizes the divisors $\{ D + \ddiv(\varphi) \ \vert \ \varphi \in \Sigma\}$, and hence is naturally identified with a subspace of $\Sym^d(\Gamma)$, where $d = \deg (D)$.  Note that $\Sym^d(\Gamma)$ carries a natural polyhedral structure. Roughly speaking, a \emph{definable subset} of $\Sym^d(\Gamma)$ is a Boolean combination of polyhedral subsets defined by inequalities on affine-linear functions with integer slopes.  See Section~\ref{sec:modules} for further details.

\begin{definition}
\label{Def:StrongTLS}
A tropical linear series $\Sigma \subseteq R(D)$ of rank $r$ is called a \emph{strong tropical linear series} if it satisfies the additional properties:
\begin{enumerate}[resume=conditions]
\item  $\vert \Sigma \vert$ is a closed and definable subset of $\vert D \vert$;
\item  there is a valuated matroid $\mathcal{V}$ of rank $r+1$ on $\Sigma$ such that, if $V \in \mathcal{V}$ is a valuated circuit, then 
$\min_{\varphi \in \Sigma} \{ \varphi + V(\varphi) \}$
is a tropical dependence.
\end{enumerate}
\end{definition}

\noindent Note that property (5) of a strong tropical linear series implies property (2) above.

\subsection{Finite generation}

In \cite{HMY12}, Haase, Musiker, and Yu show that the complete tropical linear series $R(D)$ is finitely generated as a tropical module.  However, $R(D)$ typically does not satisfy the conditions of Definition~\ref{Def:TLS}. Thus, most tropical linear series are proper submodules. Finitely generated submodules of $R(D)$ have many useful algebraic and combinatorial properties; see \cite{Luo18}, as well as Lemma~\ref{Lem:FGClosed}.  However, many proper submodules of $R(D)$ are not finitely generated. See Examples~\ref{Ex:HalfOpenInterval} and~\ref{Ex:FG}.  For tropicalizations of linear series, finite generation remains an open question. 

\begin{question} \label{Q:FG}
Are tropicalizations of linear series finitely generated as tropical modules?
\end{question}

The answer is trivially affirmative when $r = 0$.  Here we prove an affirmative answer for $r = 1$.  

\begin{theorem} \label{Thm:Rank1}
Let $\Sigma \subset R(D)$ be a tropical linear series of rank $1$. Then $\Sigma$ is finitely generated.
\end{theorem}

\noindent This theorem provides further motivation for Definition~\ref{Def:TLS}.  Optimistically, one might hope that finite generation follows from Definition~\ref{Def:TLS} for higher rank tropical linear series as well.

\begin{question} \label{Q:FG-tropicalizations}
Are all tropical linear series finitely generated as tropical modules?  If not, what about strong tropical linear series?
\end{question}

We also apply Theorem~\ref{Thm:Rank1} to describe the geometry of the spaces of divisors parametrized by rank 1 tropical linear series.

\begin{theorem}
Let $\Sigma$ be a tropical linear series of rank 1.  Then $\Sigma$ is a strong tropical linear series and hence $\vert \Sigma \vert$ is compact of pure dimension 1.
\end{theorem}

Our interest in Question~\ref{Q:FG} stems in part from the observation that a finitely generated submodule $\Sigma \subseteq R(D)$ has a minimal generating set, unique up to tropical scaling (see Lemma~\ref{Lem:UniqueGen}). Also, a finite generating set $\{ \varphi_0 , \ldots , \varphi_s \}$ determines a map to tropical projective space $\Phi \colon \Gamma \to \mathbb{TP}^s$, by
\[
\Phi (v) = [\varphi_0 (v) : \ldots : \varphi_s (v)] .
\]
If $\Sigma$ is a strong tropical linear series of rank $r$, then the image of $\Phi$ is contained in a tropical linear space of dimension $r$, corresponding to the valuated matroid on $\{ \varphi_0 , \ldots , \varphi_s \}$.  Thus, finite generation could give a useful connection to the canonical correspondence in algebraic geometry between linear series and rational maps to projective space.  Indeed, in Proposition~\ref{Prop:Harmonic}, we show that there exists a tropical modification $\widetilde{\Gamma}$ of $\Gamma$ and a balanced map $\widetilde{\Phi} : \widetilde{\Gamma} \to \mathcal{B}(M)$ extending $\Phi$.

In the case $r=1$, it follows that the only obstructions to lifting a tropical linear series of rank 1 to an algebraic linear series are the local Hurwitz obstructions (see \cite[Proposition~3.3]{ABBR2}).  Moreover, if $\Sigma$ is a tropical linear series of rank $1$ on $\Gamma$ and $K$ is an algebraically closed and  nontrivially valued field of residue characteristic zero, then there is a curve $X$ over $K$ whose skeleton has underlying metric graph $\Gamma$ (possibly with vertices of positive genus) and a linear series of rank $1$ on $X$ whose tropicalization is $\Sigma$ (see Theorem~\ref{Thm:Lifting}).  In Examples~\ref{Ex:Harmonic} and~\ref{Ex:LoopOfLoops}, we apply this to pathological examples of non-realizable rank 1 divisors appearing in the literature.  We invite the reader to compare this to Example~\ref{Ex:Matroid}, which shows that some tropical linear series of rank 2 cannot be the tropicalization of a linear series.  Indeed, a necessary condition for lifting a finitely generated strong tropical linear series is that the valuated matroid in property (5) can be chosen to be realizable.  Finding other such obstructions to lifting is an interesting open problem.

We close the paper with a section of open questions.  These include questions about finite generation, topological properties of tropical linear series, and the lifting problem.

\begin{remark}
When this paper was in the final stages of editing, we learned of a purely combinatorial notion of limit linear series on metric graphs developed simultaneously and independently by Omid Amini and Lucas Gierczak, with many similarities to the tropical linear series introduced and studied here \cite{AminiGierczak22}. There are also meaningful differences between the two approaches. For instance, Amini and Gierczak include additional combinatorial local data (rank functions on hypercubes) that encode the possibilities for the collections of slopes along tangent vectors at each vertex that are realized by functions in the given tropical submodule of $R(D)$; these are closely related to matroids and permutation arrays. Both approaches have the potential for significant applications to the study of linear series on algebraic curves, and both allow for multiple variations on the fundamental definitions. The precise relations among these definitions are not yet clear and will be explored in future work. 
\end{remark}

\medskip

{\textbf{Acknowledgments.}} Parts of this work arose from conversations at an AIM workshop in 2017.  We thank Max Kutler, Dhruv Ranganathan and Josephine Yu for helpful discussions at this workshop, as well as Omid Amini, Matt Baker, Lucas Gierczak, and Diane Maclagan for comments on an earlier draft of this paper.  The work of DJ is supported by NSF grant DMS-2054135.  The work of SP is supported by NSF grants DMS--2001502 and DMS--2053261.

\section{Divisors on Metric Graphs and Tropical Modules}
\label{Sec:Modules}

This section contains relevant background information on the divisor theory of metric graphs and tropical linear algebra.

\subsection{Divisor Theory of Metric Graphs}

A \emph{metric graph} is a compact, connected length metric space $\Gamma$ homeomorphic to a topological graph.  A finite set $V \subseteq \Gamma$  containing all points of valence different from 2 determines a graph, called a \emph{model} of $\Gamma$, with vertex set $V$ and edges corresponding to the connected components of $\Gamma \smallsetminus V$.  Note that a metric graph admits infinitely many different models.  A neighborhood of any point $v \in \Gamma$ is homeomorphic to a union of half-open line segments, glued together at their endpoints (which are identified with $v$). A \emph{tangent vector} based at $v \in \Gamma$ is a tangent vector based at $v$ in any of these line segments.

In analogy with algebraic curves, $\Div (\Gamma)$ is the free abelian group generated by points of $\Gamma$.  Elements of $\Div (\Gamma)$ are called \emph{divisors}.  In other words, a divisor on $\Gamma$ is a formal sum of points of $\Gamma$.  A divisor $D$ is called \emph{effective}, and we write $D \geq 0$, if all of its coefficients are nonnegative.  This induces a partial ordering on divisors, where $D \geq E$ if $D-E \geq 0$.

We define $\PL (\Gamma)$ to be the set of continuous, piecewise-linear functions $\varphi \colon \Gamma \to \mathbb{R}$ with integer slopes.  The set $\PL (\Gamma)$ is a semiring under the operations of pointwise minimum and pointwise addition.  Given a function $\varphi \in \PL (\Gamma)$ and a tangent vector $\zeta$ on $\Gamma$, we write $s_{\zeta} (\varphi)$ for the outgoing slope of $\varphi$ along $\zeta$.  For any point $v \in \Gamma$, the map $\ord_v \colon \PL (\Gamma) \to \mathbb{Z}$ is given by
\[
\ord_v (\varphi) := - \sum_{\zeta \in T_v (\Gamma)} s_{\zeta} (\varphi),
\]
where the sum is over all tangent vectors $\zeta$ based at $v$.  The map $\ddiv \colon \PL (\Gamma) \to \Div (\Gamma)$ is given by
\[
\ddiv (\varphi ) := \sum_{v \in \Gamma} \ord_v (\varphi) \cdot v .
\]
Note that, for $\varphi_1$ and $\varphi_2 \in \PL (\Gamma)$, we have $\ddiv (\varphi_1) = \ddiv (\varphi_2)$ if and only if the function $\varphi_1 - \varphi_2$ is constant on $\Gamma$.

\subsection{Tropical Modules}

Recall from the introduction that a \emph{tropical linear combination} of a collection of functions $\{ \varphi_0, \ldots, \varphi_r \} \subset \PL(\Gamma)$ is an expression of the form $\vartheta = \min \{ \varphi_i + a_i \}$, where $a_0, \ldots, a_r$ are real numbers.

\begin{definition}
A subset $\Sigma \subseteq \PL (\Gamma)$ is a \emph{tropical submodule} if it is closed under all tropical linear combinations.
\end{definition}

Note that tropical submodules are also called \emph{tropically convex sets} in \cite{DevelinSturmfels04, Luo18}.  For any divisor $D$ on $\Gamma$, the set
\[
R(D) := \{ \varphi \in \PL (\Gamma) \mbox{ } \vert \mbox{ } \ddiv (\varphi) + D \geq 0 \}
\]
is a tropical submodule of $\PL (\Gamma)$. 

\begin{definition}
Let $\Sigma \subseteq \PL (\Gamma)$ be a tropical submodule, and let $S \subseteq \Sigma$ be a subset.  The \emph{module} $\langle S \rangle$ \emph{generated by} $S$ is the set of all tropical linear combinations of elements of $S$.  If $\langle S \rangle = \Sigma$, we say that $S$  is a \emph{generating set} for $\Sigma$.  The tropical submodule $\Sigma$ is said to be \emph{finitely generated} if it has a finite generating set.
\end{definition}

For any divisor $D$, the tropical module $R(D)$ is finitely generated, and the minimal generating set is essentially unique up to tropical scaling \cite{HMY12}.  The same argument, which we include for completeness, gives the following lemma.

\begin{lemma}
\label{Lem:UniqueGen}
Suppose $\Sigma \subseteq \PL (\Gamma)$ is finitely generated and $\{ \varphi_0, \ldots, \varphi_r\}$ and $\{\psi_0, \ldots, \psi_s \}$ are minimal generating sets.  Then $r = s$ and there is a permutation $\sigma$ such that $\varphi_i - \psi_{\sigma(i)}$ is constant.
\end{lemma}

\noindent In other words, there is a unique minimal generating set up to addition of scalars.

\begin{proof}
Suppose $\{ \varphi_0, \ldots, \varphi_r\}$ and $\{\psi_0, \ldots, \psi_s \}$ are generating sets and $\varphi_r - \psi_j$ is not a constant for any $j$. We will show that $\{ \varphi_0, \ldots, \varphi_r\}$ is not a minimal generating set.

 Since $\{ \psi_0, \ldots, \psi_s\}$ is a generating set, there are scalars $b_0 , \ldots b_s$ such that $\varphi_r = \min \{ \psi_j + b_j \}$.  Similarly, there are scalars $a_{ij}$ such that $\psi_j = \min \{\varphi_i + a_{ij} \}$.  Note that $\varphi_r \leq \psi_j + b_j \leq \varphi_r + a_{rj} + b_j$, for all $j$.  Thus $a_{rj} + b_j \geq 0$ and, if equality holds, then $\varphi_r - \psi_j$ is the constant function $b_j$.  Since we assumed $\varphi_r - \psi_j$ is not constant, it follows that $a_{rj} + b_j > 0$.  Putting everything together, we obtain
\[
\varphi_r = \min_j \{ \min_i \{  \varphi_i +a_{ij}\} + b_j \} = \min_{i=0}^r \{ \varphi_i  + \min_j \{a_{ij} + b_j \} \} = \min_{i=0}^{r-1} \{ \varphi_i + \min_j \{ a_{ij} + b_{j} \} \}.
\]
In particular, $\varphi_r$ can be written as a tropical linear combination of the functions $\varphi_0 , \ldots , \varphi_{r-1}$, so the generating set $\{ \varphi_0 , \ldots , \varphi_r \}$ is not minimal.
\end{proof}

The following lemma says that a finitely generated tropical module has property (2) in the definition of a tropical linear series (Definition~\ref{Def:TLS}) if and only if the set of generators has it.

\begin{lemma}
\label{Lem:GeneratorDependence}
Suppose every set of $r + 2$ functions in $S = \{ \varphi_0 , \ldots , \varphi_s \} \subseteq \PL (\Gamma )$ is tropically dependent.  Then every set of $r+2$ functions in $\langle S \rangle$ is tropically dependent.
\end{lemma}

\begin{proof}
Let $\psi_i = \min_j \{ a_{ij} + \varphi_j \} \in \langle S \rangle$, for $i = 0, \ldots , r+1$, and suppose for contradiction that the set $\{ \psi_0 , \ldots , \psi_{r+1} \}$ is independent.  By Theorem~\ref{Thm:Ind}, there is a certificate of independence $\vartheta = \min \{ b_i + \psi_i \}$.  In other words, for each $i$ there exists a point $x_i \in \Gamma$ such that $b_i + \psi_i$ is the unique function that obtains the minimum at $x_i$.  Let $\varphi_{\sigma (i)}$ be a function in $S$ that obtains the minimum at $x_i$ in the expression $\psi_i = \min_j \{ a_{ij} + \varphi_j \}$.  Because $b_i + \psi_i$ is the unique function that obtains the minimum at $x_i$, we see that $b_i + a_{i\sigma (i)} < b_j + a_{j\sigma (i)}$ for all $j \neq i$.  In other words, if $j \neq i$, then $\vartheta < b_j + a_{j\sigma (i)} + \varphi_{\sigma (i)}$.  It follows that $\sigma (i) \neq \sigma (j)$ for all $i \neq j$.  We therefore see that $\vartheta' = \min \{ b_i + a_{i \sigma (i)} + \varphi_{\sigma (i)} \}$ is a certificate of independence for a set of $r+2$ functions in $S$, a contradiction. 
\end{proof}

\noindent We do not know whether there is an analogue of Lemma~\ref{Lem:GeneratorDependence} for property (5), i.e., if there is a valuated matroid on a generating set for a tropical linear series $\Sigma$ such that every valuated circuit gives rise to a tropical dependence, does it follow that this extends to a valuated matroid on $\Sigma$ with this property?

In our definition of tropical linear series, it will be important to keep track of the possible slopes of functions along a fixed tangent vector, as the function varies within a tropical module.

\begin{definition}
Let $\Gamma$ be a metric graph, $D$ a divisor on $\Gamma$, $\zeta$ a tangent vector on $\Gamma$, and $\Sigma \subseteq R(D)$ a tropical submodule.  We write
\[
s_{\zeta} (\Sigma) := (s_{\zeta}[0], \ldots , s_{\zeta}[r])
\]
for the vector of slopes $s_{\zeta} (\varphi)$, for $\varphi \in \Sigma$, ordered so that $s_{\zeta}[0] < \cdots < s_{\zeta} [r]$.  
\end{definition}

Note that the vector $s_{\zeta} (\Sigma)$  has finite length.  This is because $R(D)$ is finitely generated, and for any tangent vector $\zeta$ and function $\varphi \in \Sigma$, the slope $s_{\zeta} (\varphi)$ must agree with the slope along $\zeta$ of a function in the generating set.

\begin{lemma}
\label{Lem:Slopes}
Let $\Sigma \subseteq R(D)$ be a tropical linear series of rank $r$.  For each tangent vector $\zeta$, there are exactly $r+1$ slopes $s_{\zeta} (\varphi)$, for $\varphi \in \Sigma$.
\end{lemma}

\begin{proof}
Let $v \in \Gamma$, let $\zeta$ be a tangent vector based at $v$, and let $I \subseteq \Gamma$ be a half-open interval with one endpoint at $v$ that contains $\zeta$.  By choosing $I$ sufficiently small, we may assume that $I \smallsetminus \{ v \}$ does not intersect the support of $D$, and $s_{\eta} (\Sigma) = s_{\zeta} (\Sigma)$ for all tangent vectors $\eta$ in $I$ that are oriented away from $v$.  By property (1), if $E$ is the sum of $r$ distinct points of $I \smallsetminus \{ v \}$, then there exists $\varphi \in \Sigma$ such that $\ddiv (\varphi) + D \geq E$.  At each of the $r$ distinct points, the incoming slope of $\varphi$ must be greater than the outgoing slope, and thus the function $\varphi$ has at least $r+1$ distinct slopes on $I$.  It follows that there are at least $r+1$ slopes $s_{\zeta} (\varphi)$, for $\varphi \in \Sigma$.

On the other hand, any set of functions in $\Sigma$ with distinct slopes along $\zeta$ is tropically independent.  By property (2), therefore, there are at most $r+1$ slopes $s_{\zeta} (\varphi)$, for $\varphi \in \Sigma$.
\end{proof}

\subsection{Geometry of Tropical Modules} \label{sec:modules}

If $\Sigma \subseteq R(D)$ is a tropical submodule, we write
\[
\vert \Sigma \vert := \{ \ddiv (\varphi) + D \mbox{ } \vert \mbox{ } \varphi \in \Sigma \} .
\]
The set $\vert \Sigma \vert$ inherits the subspace topology from $\Sym^d \Gamma$, where $d = \deg (D)$. This is equivalent to the metric topology determined by $\| \ddiv (f) \|_{\infty} := \max f  - \min f$ \cite[Proposition~B.1]{Luo18}.

A choice of model for $\Gamma$, 
together with a choice of orientation of each edge, induces a polyhedral cell decomposition of $\vert D \vert := \vert R(D) \vert$ \cite{GathmannKerber08, HMY12}.  Given an open, oriented edge $e$ of length $\ell (e)$, identify $\Sym^k e$ with the open simplex $\{ \vec{x} \in \mathbb{R}^k \ \vert \ 0 < x_1 < \cdots < x_k < \ell (e) \}$.  The cells of $\vert D \vert$ are indexed by the following discrete data:
\begin{itemize}
\item  for each vertex $v$, an integer $d_v$;
\item  for each edge $e$, an integer $m_e$ and an ordered set of integers $d_e^1 , \ldots , d_e^{r_e}$.
\end{itemize}
A divisor $D' \in \vert D \vert$ belongs to the corresponding cell if:
\begin{itemize}
\item  $D'(v) = d_v$ for all vertices $v$,
\item  the restriction of $D'$ to $e$ is of the form $\sum_{i=1}^{r_e} d_i x_i$, for some $0 < x_1 < \cdots < x_{r_e} < \ell (e)$, and
\item  any $\varphi \in \PL (\Gamma)$ satisfying $\ddiv (\varphi) = D'-D$ has starting slope $m_e$ along $e$, for all edges $e$.
\end{itemize}

Let $G \subseteq \mathbb{R}$ be an additive subgroup that contains all of the edge lengths $\ell(e)$. A closed polyhedron in $\mathbb{R}^n$ is \emph{$G$-definable} if it is the intersection of finitely many halfspaces $H_i = \{ u \in \mathbb{R}^n : \langle u, v_i \rangle \geq a_i \}$ with $v_i \in \mathbb{Z}^n$ and $a_i \in G$.  For instance, the $\mathbb{Q}$-definable polyhedra are precisely the rational polyhedra. More generally, a $G$-definable subset of $\mathbb{R}^n$ is a finite Boolean combination of $G$-definable polyhedra.  The cell decomposition described above provides an expression for $\vert D \vert$ as a finite union of $G$-definable subsets of $\mathbb{R}^n$.  We say that a subset of $\vert D \vert$ is \emph{definable} if its intersection with each of these subsets is definable. 

We note some topological properties of tropical modules.

\begin{lemma}
\label{Lem:Connect}
If $\Sigma \subseteq \PL (\Gamma)$ is a tropical submodule, then $\vert \Sigma \vert$ is path-connected.
\end{lemma}

\begin{proof}
Let $\Sigma \subseteq \PL (\Gamma)$ be a tropical submodule, and let $\varphi_1 , \varphi_2 \in \Sigma$.  Choose $M$ sufficiently large so that $M + \varphi_1 (v) > \varphi_2 (v)$ and $\varphi_1 (v) < M + \varphi_2 (v)$ for all $v \in \Gamma$.  (Such an $M$ exists because $\Gamma$ is compact.)  For $0 \leq t \leq M$, consider the function
\[
\psi_t = \min \{ t + \varphi_1 , (M-t) + \varphi_2 \} .
\]
Note that $\psi_t \in \Sigma$ for all $t$, $\psi_0 = \varphi_1$, and $\psi_M = \varphi_2$.  Thus, there is a path from $\varphi_1$ to $\varphi_2$ in $\Sigma$.
\end{proof}

\begin{lemma}
\label{Lem:FGClosed}
If $\Sigma \subseteq R(D)$ is a finitely generated submodule, then $\vert \Sigma \vert$ is a closed, definable subset of $\vert D \vert$.
\end{lemma}

\begin{proof}
Let $\{ \varphi_0 , \ldots , \varphi_r \}$ be a generating set for $\Sigma$.  The map $T \colon \mathbb{R}^{r+1} \to \vert \Sigma \vert$ given by 
\[
T (a_0 , \ldots , a_r) = D + \ddiv ( \min \{ a_0 + \varphi_0 , \ldots a_r + \varphi_r \} )
\]
is continuous, piecewise-affine, and surjects onto $\vert \Sigma \vert$.  Because the theory of divisible ordered abelian groups has elimination of quantifiers \cite[Corollary~3.1.17]{Marker02}, the image of a piecewise-affine map is definable, hence $\vert \Sigma \vert$ is definable.

Since $\Gamma$ is compact, there exists an $M$ sufficiently large so that $\varphi_i < \varphi_j + M$ for all $i,j$.  Let $C \subset \mathbb{R}^{r+1}$ denote the cube $[0,M]^{r+1}$.  The restriction $T \vert_C$ surjects onto $\vert \Sigma \vert$, so $\vert \Sigma \vert$ is the continuous image of a compact set, and is therefore closed.
\end{proof}

\begin{example}
\label{Ex:HalfOpenInterval}
The hypothesis that $\Sigma$ is finitely generated is necessary for Lemma~\ref{Lem:FGClosed}.  For example, let $\Gamma$ be the unit interval, let $v$ be its left endpoint, $w$ its right endpoint, and let
\[
\Sigma = \{ \varphi \in R(v) \mbox { } \vert \mbox{ } \varphi (w) - \varphi (v) < 1 \} .
\]
Then $\Sigma$ is a tropical submodule of $R(v)$, but $\vert \Sigma \vert = \Gamma \smallsetminus \{ w \}$ is not a closed subset of $\Gamma$.
\end{example}

\subsection{Tropicalization}

Let $K$ be an algebraically closed valued field with valuation ring $R$, and let $C$ be a curve over $K$.  A \emph{semistable model} of $C$ is a flat, proper subscheme $\mathcal{C}$ over $\Spec R$, with generic fiber $\mathcal{C}_K \cong C$, whose special fiber is reduced and has no singularities other than nodes.  A semistable model determines a metric graph $\Gamma$, called the \emph{metric dual graph}, defined as follows.  The vertices of $\Gamma$ correspond to irreducible components of the special fiber, and the edges correspond to the nodes of the special fiber.  Each node in the special fiber is \'{e}tale locally isomorphic to $xy=f$ for some $f$ in the maximal ideal in $R$. The length of an edge is then defined to be $\val(f)$; this is is well-defined, independent of all choices.

For each point $v \in \Gamma$, there is a corresponding valuation $\val_v$ on the function field $K(C)$.  If $v$ is a vertex, then $\val_v (f) = -\val (a)$, where $a \in K^*$ is a constant such that $af$ is regular and nonvanishing at the generic point of the corresponding component of the special fiber.  The points in the interior of edges correspond to monomial valuations in the coordinates $x$ and $y$ near a node \'{e}tale locally isomorphic to $xy=f$.  By identifying each point of $\Gamma$ with a valuation, we obtain a map from $\Gamma$ to the Berkovich analytification $C^{\an}$.  This map is a homeomorphism onto its image, and there is a natural retraction $\Trop \colon C^{\an} \to \Gamma$.  Restricting to $K$-points of $C$ and extending linearly, one obtains a \emph{tropicalization} map
\[
\Trop \colon \Div (C) \to \Div (\Gamma) .
\]
Similarly, there is a tropicalization map
\[
\trop \colon K(C)^* \to \PL (\Gamma)
\]
given by $\trop (f) (v) = \val_v (f)$.  The two tropicalization maps are compatible in the sense that $\ddiv (\trop (f)) = \Trop (\ddiv(f))$ for all $f \in K(C)^*$.

\subsection{Valuated Matroids}

There are many equivalent definitions of matroids.  We include here the definition in terms of circuits, since valuated circuits appear in Definition~\ref{Def:StrongTLS}.

\begin{definition}
A \emph{matroid} of rank $r$ is a pair $M = (E,\mathcal{C})$ consisting of a set $E$ and a nonempty collection $\mathcal{C}$ of subsets of $E$, called \emph{circuits}, satisfying the following properties:
\begin{enumerate}
\item  no proper subset of a circuit is a circuit,
\item $\max \{ \vert A \vert \colon A \subseteq E, A \not\supseteq C \mbox{ for all } C \in \mathcal{C} \} = r$, and
\item  if $C_1 , C_2$ are circuits, $i \in C_1 \cap C_2$, and $j \in C_1 \smallsetminus C_2$, then there is a circuit $C_3 \subseteq (C_1 \cup C_2) \smallsetminus \{ i \}$ with $j \in C_3$.
\end{enumerate}
\end{definition}

Note that our definition does not assume the ground set $E$ to be finite, but does require the rank to be finite.
A typical example of a matroid is where the set $E$ is a set of vectors in a finite-dimensional vector space, and the circuits are the minimal linearly dependent subsets of $E$.  Matroids of this type are called \emph{realizable}.  If the vector space is defined over a nonarchimedean field, then the matroid has the additional structure of a valuated matroid.

The \emph{support} of a function $V \colon E \to \mathbb{R} \cup \{ \infty \}$ is $\supp(V) := \{ e \in E : V_i(e) \neq \infty \}$.

\begin{definition}
A \emph{valuated matroid} of rank $r$ is a set $E$ together with a collection $\mathcal{V}$ of functions (valuated circuits) $V \colon E \to \mathbb{R} \cup \{ \infty \}$ satisfying: \begin{enumerate}
\item $(\infty, \ldots, \infty) \not \in \mathcal{V}$;
\item  If $V \in \mathcal{V}$ and $a \in \mathbb{R}$ then $V + (a, \ldots, a) \in \mathcal{V}$;
\item  if $V_1 , V_2 \in \mathcal{V}$ and $\supp(V_1) \neq \supp(V_2)$ then $\supp(V_1) \not \subset \supp(V_2)$;
\item $\max \{ \vert A \vert \colon A \subseteq E, A \not\supseteq \supp(V) \mbox{ for all } V \in \mathcal{V} \} = r$;
\item If $V_1, V_2 \in \mathcal{V}$ and $e, e' \in E$ with $V_1(e) = V_2(e) \neq \infty$ and $V_1(e') < V_2(e')$ then there exists $V \in \mathcal{V}$ such that $V(e) = \infty$, $V(e') = V_1(e')$, and $V \geq \min \{ V_1, V_2 \}$.
\end{enumerate}
\end{definition}

Given a valuated matroid $\mathcal{V}$ of rank $r +1$ on $\{1, \ldots, n\}$, the set
\[
\mathcal{B}(\mathcal{V}) := \{ x \in \mathbb{T}\mathbb{P}^{n-1} \mbox{ } \vert \mbox { } \min_{i} \{V(i) + x_i \} \text{ occurs at least twice for all valuated circuits } V \in \mathcal{V} \}
\]
is a tropical variety of dimension $r$ and degree 1 \cite[Theorem~4.4.5]{MaclaganSturmfels}.  A \emph{tropical linear space} is a tropical variety of this form.  Let us also recall that $\mathcal{C} = \{ \supp(V) : V \in \mathcal{V} \}$ is the set of circuits in an ordinary matroid, called the \emph{underlying matroid} of $\mathcal{V}$. The Bergman fan of  the underlying matroid of $\mathcal{V}$ is the recession fan of the tropical linear space $\mathcal{B}(\mathcal{V})$.

\section{Examples}

Before proving our main results, we present examples to illustrate the relations between complete linear series $R(D)$, tropicalizations of algebraic linear series, and our notion of tropical linear series.

\begin{example}
\label{Ex:Lollipop}
If $\Gamma$ is a tree, then $R(D)$ is a tropical linear series of rank $\deg (D)$ for any effective divisor $D$ on $\Gamma$.  Similarly, if $\Gamma$ is a circle, then $R(D)$ is a tropical linear series of rank $\deg(D)-1$ for any non-trivial effective divisor $D$ on $\Gamma$.  Beyond these cases, it is typical for a graph $\Gamma$ to have divisors $D$ such that $R(D)$ is not a tropical linear series.  For example, consider the ``lollipop'' graph, pictured in Figure~\ref{Fig:Lollipop}, consisting of a line segment and a circle meeting at a point $w$.  Let $D = mw$ for some integer $m > 0$.  If $\zeta$ is the leftward tangent vector based at $w$, then
\[
s_{\zeta} (R(D)) = (0,1,\ldots,m).
\]
In particular, there are $m+1$ slopes $s_{\zeta} (\varphi)$, as $\varphi$ ranges over $R(D)$.  On the other hand, if $\eta$ is the upward tangent vector based at the point $w$, then
\[
s_{\eta} (R(D)) = (0,1,\ldots,m-1).
\]
In particular, there are only $m$ slopes $s_{\eta} (\varphi)$, as $\varphi$ ranges over $R(D)$.  By Lemma~\ref{Lem:Slopes}, because the number of slopes does not remain constant as the tangent vector varies, $R(D)$ cannot be a tropical linear series.

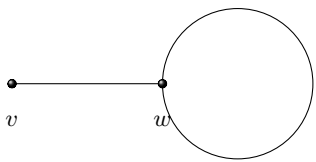
\begin{figure}[h]
\begin{tikzpicture}

\draw (0,4)--(2,4);
\draw (3,4) circle (1);
\draw [ball color=black] (0,4) circle (0.55mm);
\draw [ball color=black] (2,4) circle (0.55mm);
\draw (0,3.5) node {{\small $v$}};
\draw (2,3.5) node {{\small $w$}};

\end{tikzpicture}
\caption{The lollipop graph.}
\label{Fig:Lollipop}
\end{figure}

This example also illustrates that a finite subset of $R(D)$ is not necessarily a matroid under tropical dependence.  To see this, let $m=2$, and let $\varphi_0 \in R(D)$ be a constant function.  Let $\varphi_1 \in R(D)$ be a function with $s_{\zeta} (\varphi_1 ) = 1$, $\varphi_2 \in R(D)$ a function with $s_{\zeta} (\varphi_2 ) = 2$, and $\varphi_3 \in R(D)$ a function with $s_{\eta} (\varphi_3) = 1$.  Note that the restrictions of $\varphi_1$ and $\varphi_2$ to the circle are constant functions, and the restriction of $\varphi_3$ to the stem of the lollipop is a constant function.  The constant function $\varphi_0$ can be expressed as a tropical linear combination of either $\varphi_1$ and $\varphi_3$ or $\varphi_2$ and $\varphi_3$, and thus the sets $C_1 = \{ \varphi_0 , \varphi_1 , \varphi_3 \}, C_2 = \{ \varphi_0 , \varphi_2 , \varphi_3 \}$ are minimal tropically dependent sets.  However, the set $C_1 \cup C_2 \smallsetminus \{ \varphi_3 \} = \{ \varphi_0 , \varphi_1 , \varphi_2 \}$ is tropically independent, since the three functions have distinct slopes along the tangent vector $\zeta$.
\end{example}

\begin{example}
\label{Ex:Barbell}
Let $\Gamma$ be the barbell graph pictured in Figure~\ref{Fig:Barbell}, and consider the canonical divisor $K_{\Gamma} = v+w$.  As in Example~\ref{Ex:Lollipop}, the module $R(K_{\Gamma})$ is not a tropical linear series, because there are 3 possible slopes of functions in $R(K_{\Gamma})$ along the bridge between the loops, but only 2 possible slopes along tangent vectors in the loops.  We will show that $R(K_{\Gamma})$ contains a unique tropical linear series $\Sigma$ of rank 1.  Because the tropicalization of a linear series of degree $d$ and rank $r$ is a tropical linear series of the same degree and rank, $\Sigma$ must be the tropicalization of the canonical linear series on any curve of genus 2 with skeleton $\Gamma$.  Indeed, the characterization of $\Sigma$ that we deduce from Definition~\ref{Def:TLS} agrees with the realizability locus for canonical divisors determined in \cite[Example~6.4]{MUW21}.

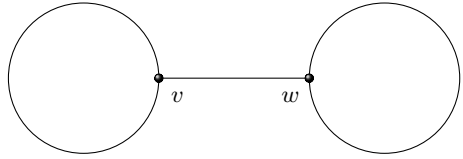
\begin{figure}[h]
\begin{tikzpicture}

\draw (0,4)--(2,4);
\draw (3,4) circle (1);
\draw (-1,4) circle (1);
\draw [ball color=black] (0,4) circle (0.55mm);
\draw [ball color=black] (2,4) circle (0.55mm);
\draw (0.25,3.75) node {{\small $v$}};
\draw (1.75,3.75) node {{\small $w$}};

\end{tikzpicture}
\caption{The barbell graph.}
\label{Fig:Barbell}
\end{figure}

Let $\Sigma \subseteq R(K_{\Gamma})$ be a tropical linear series of rank 1, let $\zeta$ be a rightward tangent vector based at a point on the bridge between the two loops, and let $x \in \Gamma$ be a point.  If $x \neq w$ is on the right loop, then up to tropical scaling there exists a unique function $\varphi \in R(K_{\Gamma})$ such that $\ddiv (\varphi) + K_{\Gamma} \geq x$.  The slopes of this function are pictured on the right of Figure~\ref{Fig:BarbellSlopes}, where each edge is oriented from left to right.  Since $\Sigma$ has rank 1, by the Baker-Norine rank condition, $\varphi \in \Sigma$.  Note that $s_{\zeta} (\varphi) = 1$.  Similarly, if $x \neq v$ is on the left loop, then up to tropical scaling there exists a unique function $\varphi \in R(K_{\Gamma})$ such that $\ddiv (\varphi) + K_{\Gamma} \geq x$.  This function, pictured on the left of Figure~\ref{Fig:BarbellSlopes}, must be in $\Sigma$, and $s_{\zeta} (\varphi) = -1$.  Together, this shows that $s_{\zeta} (\Sigma)$ must be $(-1,1)$ for all rightward tangent vectors $\zeta$ on the bridge.  Finally, if $x$ is on the bridge, then among the functions $\varphi \in R(K_{\Gamma})$ with $\ddiv (\varphi) + K_{\Gamma} \geq x$, there is a unique one satisfying $s_{\zeta} (\varphi) = \pm 1$ for all tangent vectors $\zeta$ on the bridge.  This function is pictured on the bottom in Figure~\ref{Fig:BarbellSlopes}.  By the Baker-Norine rank condition, all three types of functions must be contained in $\Sigma$, and by the slope sequence condition, no other functions in $R(K_{\Gamma})$ may be contained in $\Sigma$.

\begin{figure}[h!]
\begin{tikzpicture}

\draw (0,4)--(2,4);
\draw (3,4) circle (1);
\draw (-1,4) circle (1);
\draw [ball color=black] (-1,5) circle (0.55mm);
\draw [ball color=black] (-1,3) circle (0.55mm);
\draw (-1,5.25) node {{\small $x$}};
\draw (0.25,3.75) node {{\small $v$}};
\draw (1.75,3.75) node {{\small $w$}};
\draw (-2.25,4) node {{\small $0$}};
\draw (-0.5,4.5) node {{\small $-1$}};
\draw (-0.5,3.5) node {{\small $-1$}};
\draw (1,4.25) node {{\small $-1$}};
\draw (2.25,4) node {{\small $0$}};

\draw (7,4)--(9,4);
\draw (10,4) circle (1);
\draw (6,4) circle (1);
\draw [ball color=black] (10,5) circle (0.55mm);
\draw [ball color=black] (10,3) circle (0.55mm);
\draw (10,5.25) node {{\small $x$}};
\draw (7.25,3.75) node {{\small $v$}};
\draw (8.75,3.75) node {{\small $w$}};
\draw (11.25,4) node {{\small $0$}};
\draw (9.5,4.5) node {{\small $1$}};
\draw (9.5,3.5) node {{\small $1$}};
\draw (8,4.25) node {{\small $1$}};
\draw (6.75,4) node {{\small $0$}};

\draw (3.5,1)--(5.5,1);
\draw (2.5,1) circle (1);
\draw (6.5,1) circle (1);
\draw [ball color=black] (4.5,1) circle (0.55mm);
\draw (4.5,1.25) node {{\small $2$}};
\draw (4.5,0.75) node {{\small $x$}};
\draw (3.75,0.75) node {{\small $v$}};
\draw (5.25,0.75) node {{\small $w$}};
\draw (3.25,1) node {{\small $0$}};
\draw (4,1.25) node {{\small $1$}};
\draw (5,1.25) node {{\small $-1$}};
\draw (5.75,1) node {{\small $0$}};

\end{tikzpicture}
\caption{The 3 types of functions in $\Sigma$.}
\label{Fig:BarbellSlopes}
\end{figure}
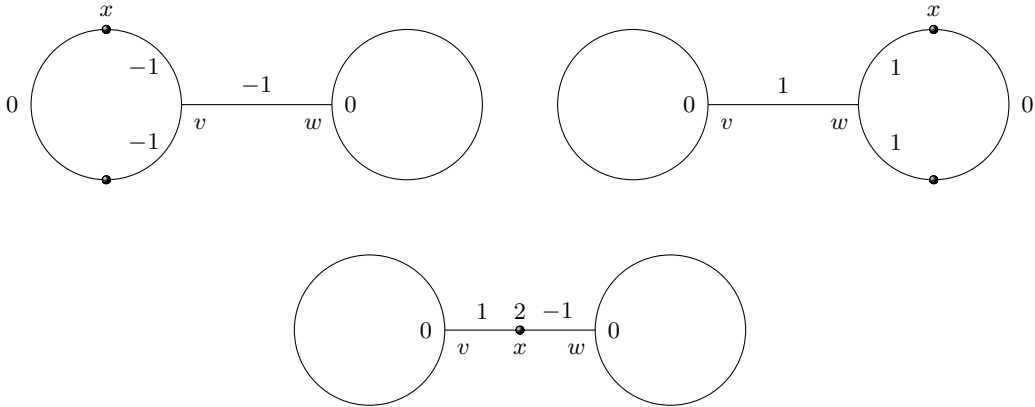
\end{example}

\begin{example}
\label{Ex:Interval}
Let $\Gamma$ be an interval with left endpoint $x$ and right endpoint $y$, and let $D$ be a divisor of degree 2 on $\Gamma$.  Rank 1 tropical linear series $\Sigma \subseteq R(D)$ are classified in \cite[Example~6.8]{M23}.  We reproduce this classification here.  Note that all divisors of degree 2 on $\Gamma$ are linearly equivalent; so there exists $\varphi \in \PL (\Gamma)$ such that $\ddiv (\varphi) = D-2x$, and the map from $R(D)$ to $R(2x)$ given by $\psi \mapsto \psi + \varphi$ is an isomorphism of tropical modules.  Thus we may assume $D=2x$.

Let $\varphi_0 \in \Sigma$ satisfy $s_x (\varphi_0) = s_x [0]$, and let $\varphi_1 \in \Sigma$ satisfy $s_y (\varphi_1) = s_y [1]$, where both tangent vectors point to the right.  Let $w_0 = \ddiv(\varphi_0) + D - x$ and $w_1 = \ddiv (\varphi_1) + D-y$.  If $w_0$ is to the left of $w_1$ or $w_0 = w_1$, we show that $\Sigma = \langle \varphi_0 , \varphi_1 \rangle$.  Indeed, if $\psi \in \Sigma$, then there is a tropical dependence $\vartheta = \min \{ a_0 + \varphi_0 , a_1 + \varphi_1 , a_2 + \psi \}$.  By assumption, however, the slope of $\varphi_1$ is strictly greater than that of $\varphi_0$ at every point, so there is at most one point where $a_0 + \varphi_0$ and $a_1 + \varphi_1$ agree.  In order for the minimum to be achieved at least twice at every point, it follows that $a_2 + \psi = \min \{ a_0 + \varphi_0 , a_1 + \varphi_1 \}$, hence $\psi \in  \langle \varphi_0 , \varphi_1 \rangle$.

On the other hand, suppose $w_0$ is to the right of $w_1$.  Then there is a distinguished point $z$ in this region such that the tangent vector $s_{\zeta} (\Sigma)$ is equal to $(1,2)$ for all tangent vectors $\zeta$ to the left of $z$, and equal to $(0,1)$ for all tangent vectors $\zeta$ to the right of $z$.  In fact, $z$ must be the midpoint between $w_1$ and $w_0$ \cite[Example~6.8]{M23}.  By taking a tropical linear combination of functions that agree at $z$, one with incoming slope 2 and one with outgoing slope 0, we obtain a function $\varphi_2 \in \Sigma$ with slope 2 everywhere to the left of $z$ and slope 0 everywhere to the right of $z$.

We claim that $\Sigma = \langle \varphi_0 , \varphi_1 , \varphi_2 \rangle$.  Let $\psi \in \Sigma$.  If $\ddiv (\psi) +2x = w_0 + w_1$, then up to scaling $\psi = \min \{  a_0 + \varphi_0 , a_1 + \varphi_1 \}$, where the coefficients are chosen so that the two functions agree at $w_1$.  Otherwise, the support of $\ddiv (\psi) +2x$ contains a point $w \neq w_0 , w_1$.  If $w$ is to the left of $z$, then consider the tropically dependent set $\{ \varphi_0 , \varphi_2 , \psi \}$.  Since $\psi$ is nonlinear at $w$ and the other two functions have distinct slopes, we see that all three functions must obtain the minimum at $w$.  Moreover, since $w \neq w_1$, there is at most one other point where the functions $\varphi_0$ and $\varphi_2$ agree.  It follows that $\psi \in \langle \varphi_0 , \varphi_2 \rangle$.  By a similar argument, if $w$ is to the right of $z$, then $\psi \in \langle \varphi_1 , \varphi_2 \rangle$. 
\end{example}

Our next example shows that tropical submodules of $R(D)$ satisfying properties (1), (3), and (4) from Definitions~\ref{Def:TLS} and~\ref{Def:StrongTLS} can fail to be finitely generated, even for $r = 1$. This highlights the essential importance of tropical dependence and independence in the theory of tropical linear series.

\begin{example}
\label{Ex:FG}
Let $\Gamma$ be a closed interval of length 2, let $v$ be the midpoint of the interval, and let $D=2v$.  Consider the tropical submodule $\Sigma \subseteq R(D)$ consisting of functions $\varphi_{xy}$ as pictured in Figure~\ref{Fig:InfGen}, with the condition that $x+y \leq 1$.  Note that $\Sigma$ satisfies every condition of Definition~\ref{Def:TLS} for a tropical linear series of rank 1 except for condition (2).  To be precise, for every point $x \in \Gamma$, there is a function $\varphi \in \Sigma$ such that $\ddiv (\varphi) + 2v \geq x$.  Note also that the set $\vert \Sigma \vert$ is a definable closed subset of $\Sym^2 \Gamma$.  Moreover, there are exactly 2 possible slopes at every tangent vector.  However, the functions $\varphi_{10}, \varphi_{01},$ and $\varphi_{\frac{1}{2} \frac{1}{2}}$ are tropically independent.  

Indeed, following the classification of rank 1 tropical linear subseries of $R(D)$ from Example~\ref{Ex:Interval}, we see that the unique function $\varphi \in \Sigma$ such that the support of $\ddiv (\varphi) + D$ contains the left endpoint is $\varphi_{10}$.  Similarly, the unique function $\varphi \in \Sigma$ such that the support of $\ddiv (\varphi) + D$ contains the right endpoint is  $\varphi_{01}$.  In Example~\ref{Ex:Interval}, if a rank 1 tropical linear series contains $\varphi_{10}$ and $\varphi_{01}$, then $w_0 = w_1 = v$, and the tropical linear series is generated by $\varphi_{10}$ and $\varphi_{01}$.  The tropical submodule $\Sigma$, however, is not generated by these two functions, so it cannot be a tropical linear series.

We now show that $\Sigma$ is not finitely generated.  Consider the tropical linear combination
\[
\varphi_{xy} = \min \{ a_{x'y'} + \varphi_{x'y'} \} .
\]
If $a_{x'y'} + \varphi_{x'y'}$ is a function that obtains the minimum at the point $v$, then we have $x \leq x'$ and $y \leq y'$.  It follows that the functions $\varphi_{xy}$ with $x+y=1$ cannot be written as tropical linear combinations of other functions in this submodule.  Hence, this submodule is not finitely generated.

\begin{figure}[h!]
\begin{tikzpicture}

\draw (0,4)--(2,4);
\draw (2,4)--(4,2);
\draw (4,2) -- (5,3);
\draw (5,3) -- (8,3);

\draw (0,0) -- (8,0);
\draw [ball color=black] (4,0) circle (0.55mm);
\draw (4,0.5) node {{\tiny $v$}};
\draw[<->] (2,-0.5)--(4,-0.5);
\draw[<->] (4,-0.5)--(5,-0.5);
\draw (3,-0.75) node {{\tiny $x$}};
\draw (4.5,-0.75) node {{\tiny $y$}};

\end{tikzpicture}
\caption{The graph of a function $\varphi_{xy}$ in $R(2v)$.}
\label{Fig:InfGen}
\end{figure}
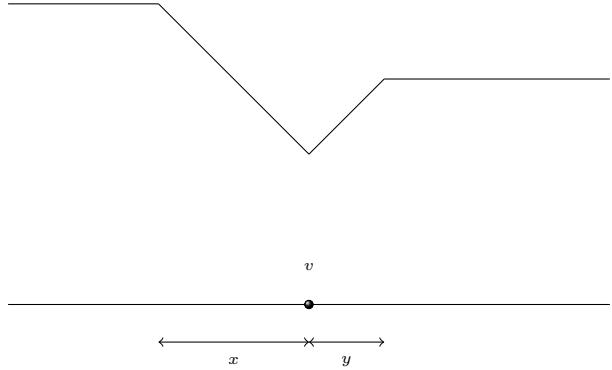
\end{example}

\begin{example}
\label{Ex:Harmonic}
One of the central problems in the theory of divisors on metric graphs is that of realizability:  given a divisor of rank $r$ on a metric graph, is it the tropicalization of a divisor of rank $r$ on an algebraic curve?  It is well known that the answer is sometimes negative.   One such example appears in \cite[Example~5.13]{ABBR2}, attributed to Ye Luo.  Consider the divisor $D = p+q+s$ on the graph $\Gamma$ pictured in Figure~\ref{Fig:NonRealize}, where all edge lengths are equal.  This divisor has rank 1, but there does not exist a degree 3 harmonic morphism from $\Gamma$ to a tree.  From this it follows that $D$ cannot be the tropicalization of a divisor of positive rank on an algebraic curve.

\begin{figure}[h]
\begin{tikzpicture}

\draw (0,0) circle (1);
\draw (0,1)--(0,2);
\draw (0,1)to[out=45,in=315](0,2);
\draw (0,1)to[out=135,in=225](0,2);
\draw (0.87,-0.5)--(1.74,-1);
\draw (0.87,-0.5)to[out=0,in=90](1.74,-1);
\draw (0.87,-0.5)to[out=270,in=180](1.74,-1);
\draw (-0.87,-0.5)--(-1.74,-1);
\draw (-0.87,-0.5)to[out=180,in=90](-1.74,-1);
\draw (-0.87,-0.5)to[out=270,in=0](-1.74,-1);
\draw [ball color=black] (0,1) circle (0.55mm);
\draw [ball color=black] (0.87,-0.5) circle (0.55mm);
\draw [ball color=black] (-0.87,-0.5) circle (0.55mm);
\draw [ball color=black] (0,2) circle (0.55mm);
\draw [ball color=black] (1.74,-1) circle (0.55mm);
\draw [ball color=black] (-1.74,-1) circle (0.55mm);
\draw (0,0.75) node {{\small $p$}};
\draw (0.62,-0.375) node {{\small $q$}};
\draw (-0.62,-0.375) node {{\small $s$}};
\draw (0,2.25) node {{\small $x$}};
\draw (2,-1.25) node {{\small $y$}};
\draw (-2,-1.25) node {{\small $z$}};

\end{tikzpicture}
\caption{Luo's example of a non-realizable divisor of positive rank.}
\label{Fig:NonRealize}
\end{figure}
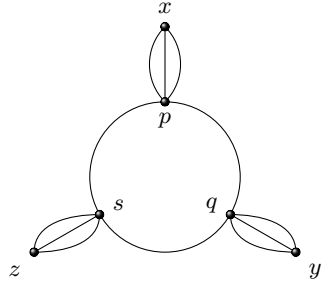

Here, we show that there is no tropical linear series $\Sigma \subseteq R(D)$ of rank 1.  This provides an alternative obstruction to lifting $D$ to a divisor of positive rank on an algebraic curve.  Indeed, these two obstructions are related; we will explore this in greater detail in Proposition~\ref{Prop:Harmonic}.

Let $\zeta$ be a rightward tangent vector on the edge between $s$ and $q$.  Up to tropical scaling there exists a unique function $\varphi_x \in R(D)$ such that $\ddiv (\varphi_x) + D \geq x$.  Note that $s_{\zeta} (\varphi_x) = 0$.  Similarly, up to tropical scaling there exists a unique function $\varphi_y \in R(D)$ such that $\ddiv (\varphi) + D \geq y$  and a unique function $\varphi_z \in R(D)$ such that $\ddiv (\varphi_z) + D \geq z$.  We have $s_{\zeta} (\varphi_y) = 1$ and $s_{\zeta} (\varphi_z) = -1$.  By the Baker-Norine rank condition, a rank 1 tropical linear series $\Sigma \subseteq R(D)$ would have to contain all 3 functions $\varphi_x , \varphi_y ,$ and $\varphi_z$.  However, because $s_{\zeta} (\varphi_x), s_{\zeta} (\varphi_y),$ and $s_{\zeta} (\varphi_z)$ are distinct, the 3 functions $\varphi_x , \varphi_y ,$ and $\varphi_z$ are independent, and thus cannot be contained in a rank 1 tropical linear series.  Thus, there is no tropical linear series $\Sigma \subseteq R(D)$ of rank 1. 
\end{example}

\begin{example}
\label{Ex:LoopOfLoops}
Another interesting example is the genus 4 loop of loops $\Gamma$, pictured in Figure~\ref{Fig:LOL}.  If $\ell_1 > \ell_2 > \ell_3$ and $\ell_2 + \ell_3 > \ell_1$, then $\Gamma$ has infinitely many divisor classes of degree 3 and rank 1, and no divisors of degree 2 and rank 1 \cite[Theorem~1.2]{LPP12}. 
 This contrasts with a well-known result for algebraic curves; by Martens's Theorem, a non-hyperelliptic curve of genus at least 4 has only finitely many divisors of degree 3 and rank 1 \cite[Theorem~1]{Martens67}.  We show that there is a unique divisor class $[D]$ on $\Gamma$ of rank 1 and degree 3 with the property that $R(D)$ contains a rank 1 tropical linear series.  It follows that $[D]$ is the only divisor class on $\Gamma$ that could be the tropicalization of a divisor class of degree 3 and rank 1 on an algebraic curve.  

\begin{figure}[h]
\begin{tikzpicture}

\draw (0,0) circle (0.5);
\draw (2,0) circle (0.5);
\draw (1,1.73) circle (0.5);
\draw (0.5,0)--(1.5,0);
\draw (0.25,0.43)--(0.75,1.3);
\draw (1.75,0.43)--(1.25,1.3);
\draw [ball color=black] (-0.5,0) circle (0.55mm);
\draw [ball color=black] (2.5,0) circle (0.55mm);
\draw [ball color=black] (1,2.23) circle (0.55mm);
\draw [ball color=black] (0.25,0.43) circle (0.55mm);
\draw [ball color=black] (0.5,0) circle (0.55mm);
\draw [ball color=black] (1.25,1.3) circle (0.55mm);
\draw (0.25,0.7) node {{\small $v_1$}};
\draw (0.6,-0.3) node {{\small $w_3$}};
\draw (1.5,1.1) node {{\small $v_2$}};

\draw (0.7,0.85) node {{\small $\ell_1$}};
\draw (1.3,0.85) node {{\small $\ell_2$}};
\draw (1,0.2) node {{\small $\ell_3$}};

\draw (-0.8,0) node {{\small $u_2$}};
\draw (2.8,0) node {{\small $u_1$}};
\draw (1,2.43) node {{\small $u_3$}};

\end{tikzpicture}
\caption{The loop of loops of genus 4}
\label{Fig:LOL}
\end{figure}
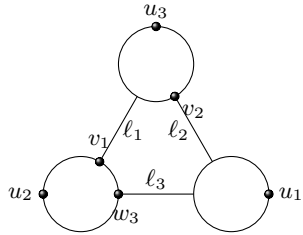

As in the proof of \cite[Theorem~1.2]{LPP12}, let $D$ be a divisor of the form $v_1 + w_3 + w$, where $w$ is a point on the edge of length $\ell_2$.  We let $x$ be the distance from $w$ to $v_2$, and assume that $x \geq \ell_1 - \ell_2$.  Let $\Sigma \subseteq R(D)$ be a tropical linear series of rank 1.  We will use the existence of $\Sigma$ to solve for $x$, thus proving that $D$ is unique.  Up to tropical scaling, for $i=1,2,3$ there exists a unique function $\varphi_i \in R(D)$ such that $\ddiv (\varphi_i) + D \geq u_i$.  The functions $\varphi_i$ are pictured in Figure~\ref{Fig:FunctionsOnLOL}.  By the Baker-Norine rank condition, $\varphi_i \in \Sigma$ for all $i$.  Thus, the set $\{ \varphi_1 , \varphi_2 , \varphi_3 \}$ must be tropically dependent.  Now, consider the point on the edge of length $\ell_3$ in the support of $\ddiv (\varphi_3) + D$, which has distance $\ell_1 - x$ from $w_3$ .  No two of the functions $\varphi_i$ agree in a neighborhood of this point, so in the tropical dependence, all three functions must obtain the minimum at this point.  Similarly, if we consider the point on the edge of length $\ell_1$ in the support of $\ddiv (\varphi_1) + D$, it has distance $\ell_3 - \ell_2 + x$ from $v_1$, and we see that all three functions must obtain the minimum in a neighborhood of this point.  The tropical dependence is pictured in Figure~\ref{Fig:LOLDependence}, with the loops labeled by the functions that obtain the minimum on them.  Since the function $\varphi_2$ takes the same value at these two points, it follows that $\ell_1 - x = \ell_3 - \ell_2 +x$.  Equivalently, $x = \frac{1}{2}(\ell_1 + \ell_2 - \ell_3)$.

\begin{figure}[h]
\begin{tikzpicture}

\draw (0,0) circle (0.5);
\draw (2,0) circle (0.5);
\draw (1,1.73) circle (0.5);
\draw (0.5,0)--(1.5,0);
\draw (0.25,0.43)--(0.75,1.3);
\draw (1.75,0.43)--(1.25,1.3);
\draw [ball color=white] (0.5,0) circle (0.55mm);
\draw [ball color=white] (0.25,0.43) circle (0.55mm);
\draw [ball color=white] (1.5,0.9) circle (0.55mm);
\draw [ball color=black] (2.5,0) circle (0.55mm);
\draw [ball color=black] (0.5,0.9) circle (0.55mm);
\draw [ball color=black] (2.25,-0.43) circle (0.55mm);

\draw (1,-0.3) node {{\small $1$}};
\draw (0.3,0.85) node {{\small $1$}};
\draw (1.7,0.85) node {{\small $1$}};
\draw (2,0) node {{\small $1$}};
\draw (1,-0.7) node {{\small $\varphi_1$}};

\draw (4,0) circle (0.5);
\draw (6,0) circle (0.5);
\draw (5,1.73) circle (0.5);
\draw (4.5,0)--(5.5,0);
\draw (4.25,0.43)--(4.75,1.3);
\draw (5.75,0.43)--(5.25,1.3);
\draw [ball color=white] (4.5,0) circle (0.55mm);
\draw [ball color=white] (4.25,0.43) circle (0.55mm);
\draw [ball color=black] (5.5,0.9) circle (0.55mm);
\draw [ball color=black] (3.5,0) circle (0.55mm);
\draw [ball color=black] (3.75,-0.43) circle (0.55mm);

\draw (4,0) node {{\small $1$}};
\draw (5,-0.7) node {{\small $\varphi_2$}};

\draw (8,0) circle (0.5);
\draw (10,0) circle (0.5);
\draw (9,1.73) circle (0.5);
\draw (8.5,0)--(9.5,0);
\draw (8.25,0.43)--(8.75,1.3);
\draw (9.75,0.43)--(9.25,1.3);
\draw [ball color=white] (8.5,0) circle (0.55mm);
\draw [ball color=white] (8.25,0.43) circle (0.55mm);
\draw [ball color=white] (9.5,0.9) circle (0.55mm);
\draw [ball color=black] (9,0) circle (0.55mm);
\draw [ball color=black] (8.9,2.23) circle (0.55mm);
\draw [ball color=black] (9.1,2.23) circle (0.55mm);

\draw (8.75,-0.3) node {{\small $1$}};
\draw (8.3,1) node {{\small $1$}};
\draw (9.6,1.2) node {{\small $1$}};
\draw (9,1.73) node {{\small $1$}};
\draw (9,-0.7) node {{\small $\varphi_3$}};

\end{tikzpicture}
\caption{The functions $\varphi_1 , \varphi_2,$ and $\varphi_3$ of Example~\ref{Ex:LoopOfLoops}.}
\label{Fig:FunctionsOnLOL}
\end{figure}
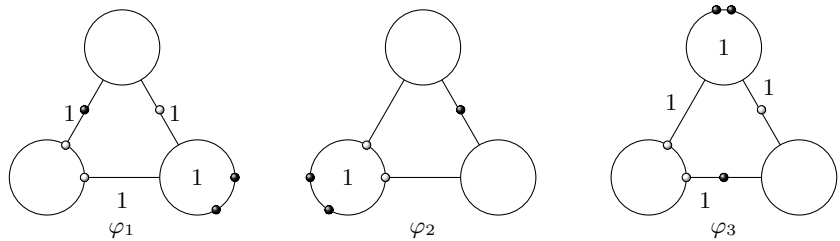

\begin{figure}[h]
\begin{tikzpicture}

\draw (0,0) circle (0.5);
\draw (2,0) circle (0.5);
\draw (1,1.73) circle (0.5);
\draw (0.5,0)--(1.5,0);
\draw (0.25,0.43)--(0.75,1.3);
\draw (1.75,0.43)--(1.25,1.3);
\draw [ball color=white] (0.5,0) circle (0.55mm);
\draw [ball color=white] (0.25,0.43) circle (0.55mm);
\draw [ball color=black] (1.5,0.9) circle (0.55mm);
\draw [ball color=black] (0.5,0.9) circle (0.55mm);
\draw [ball color=black] (1,0) circle (0.55mm);

\draw (0.75,-0.3) node {{\small $1$}};
\draw (0.3,0.85) node {{\small $1$}};

\draw (0,0) node {{\small $\varphi_1 , \varphi_3$}};
\draw (2,0) node {{\small $\varphi_2 , \varphi_3$}};
\draw (1,1.73) node {{\small $\varphi_1 , \varphi_2$}};

\end{tikzpicture}
\caption{The dependence between $\varphi_1, \varphi_2,$ and $\varphi_3$.}
\label{Fig:LOLDependence}
\end{figure}
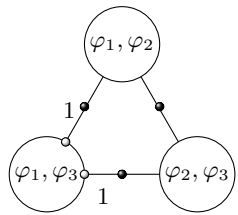

\end{example}

\begin{example}
\label{Ex:Matroid}
We now construct a family of examples of rank 2 tropical linear series.  Among them are many examples of tropical linear series that cannot be the tropicalization of a linear series on an algebraic curve.  Our treatment is inspired by \cite{Cartwright15}, and follows closely the constructions therein.

Let $M = (E,\mathcal{C})$ be a simple, rank 3 matroid with $E$ finite.  The size of every circuit in such a matroid is either 3 or 4.  A rank 2 \emph{flat} of such a matroid is a set $f \subseteq E$ with $\vert f \vert \geq 2$, such that every circuit $C$ satisfying $\vert C \vert = 3$ and $\vert C \cap f \vert \geq 2$ is contained in $f$.  The \emph{Levi graph} $\Gamma_M$ of $M$ is a bipartite graph with vertex set $V(\Gamma_M) = E \coprod F$, where $E$ is the set of elements of $M$ and $F$ is the set of rank 2 flats, and where $e \in E$ is adjacent to $f \in F$ if $e \in f$.  For simplicity, we assume all the edges of $\Gamma_M$ have length 1.  The divisor $D_M = \sum_{e \in E} e$ is shown to have rank 2 in \cite[Proposition~2.2]{Cartwright15}.  Our goal is to construct a rank 2 tropical linear series $\Sigma_M \subseteq R(D_M)$.  We will then prove an analogue of \cite[Theorem~3.4]{Cartwright15}, that if there exists a curve $C$ over a nonarchimedean field $K$ with skeleton $\Gamma_M$, and a rank 2 linear series on $C$ that specializes to $\Sigma_M$, then the matroid $M$ is realizable over $K$.  It follows that, if $M$ is non-realizable, then $\Sigma_M$ cannot be the tropicalization of a linear series on an algebraic curve.  Since there are many examples of non-realizable rank 3 matroids, this construction produces many examples of non-realizable tropical linear series.

We now construct the tropical linear series $\Sigma_M$.  For each element $e \in E$, let $\varphi_e$ be the function whose restriction to each edge is linear, and takes the following values at vertices:
\begin{displaymath}
\varphi_e (v) = \left\{ \begin{array}{ll}
2 & \textrm{if $v=e$,} \\
1 & \textrm{if $v=f \in F, f \ni e$,}\\
0 & \textrm{otherwise.}
\end{array} \right.
\end{displaymath}

\begin{proposition}
The tropical submodule $\Sigma_M = \langle \varphi_e \mbox{ } \vert \mbox{ } e \in E \rangle$ is a tropical linear series of rank 2.
\end{proposition}

\begin{proof}
The proof of \cite[Proposition~2.2]{Cartwright15} is a careful case analysis, showing that for any effective divisor $D'$ of degree 2 on $\Gamma_M$, there is a function $\varphi \in R(D_M)$ such that $\ddiv(\varphi) + D_M \geq D'$.  The functions $\varphi$ constructed in this proof are all tropical linear combinations of the functions $\varphi_e$.  This demonstrates property (1).

By Lemma~\ref{Lem:GeneratorDependence}, it suffices to check that property (2) holds for the functions $\varphi_e$.  We prove the stronger statement that, if $C \subseteq E$ is a circuit, then $\vartheta = \min_{e \in C} \{ \varphi_e \}$ is a dependence.  First, consider the case where $\vert C \vert = 4$.  In this case, we see that $\vartheta = \min_{e \in C} \{ \varphi_e \}$ is identically zero, and on an edge adjacent to a flat $f$, the minimum is achieved at least twice, by the functions $\varphi_e$ such that $e \notin f$.  The tropical linear combination $\vartheta$ is therefore a dependence.

Next, consider the case where $\vert C \vert =3$.  Then there exists a unique rank 2 flat $f$ such that $C \subseteq f$.  In this case, since every pair of elements is contained in a unique rank 2 flat, on each edge other than those adjacent to $f$, $\vartheta$ is identically zero.  If such an edge is adjacent to an element $e$, the minimum is achieved at least twice, by the functions $\varphi_{e'}$ with $e' \neq e$.  Similarly, on an edge from an element $e$ to the flat $f$, $\vartheta$ has slope 1.  Again, the minimum is achieved at least twice, by the functions $\varphi_{e'}$ with $e' \neq e$.

We now prove property (3).  For any flat $f \in F$, let $\varphi_f = \min_{e \in f} \{ \varphi_e \}$.  We have $\varphi_f (f) = 1$, and $\varphi_f (v) = 0$ for all vertices $v \neq f$.  Let $\Upsilon = \langle \varphi_f \mbox{ } \vert \mbox{ } f \in F \rangle$, and for each element $e \in E$, let $\Upsilon_e = \langle \varphi_e , \varphi_f \mbox{ } \vert \mbox{ } f \not\ni e \rangle$.  If $\zeta$ is a tangent vector on $\Gamma_M$ directed toward an element $e \in E$, then $\Upsilon \subseteq \{ \varphi \in \Sigma \mbox{ } \vert \mbox{ } s_{\zeta} (\varphi) \leq 0 \}$.  Similarly, if $\zeta$ is a tangent vector on $\Gamma_M$ directed away from an element $e \in E$, then $\Upsilon_e \subseteq \{ \varphi \in \Sigma \mbox{ } \vert \mbox{ } s_{\zeta} (\varphi) \leq 0 \}$.  It is not too difficult to verify that $\Upsilon$ and $\Upsilon_e$ are rank 1 tropical linear series.  Since $\Sigma_M$ satisfies properties (1)-(3), it is a rank 2 tropical linear series.
\end{proof}

Since $\Sigma_M$ is finitely generated, by Lemma~\ref{Lem:FGClosed}, $\vert \Sigma_M \vert$ is a closed definable subset of $\vert D_M \vert$.  Thus, $\Sigma_M$ satisfies property (4) of a strong tropical linear series as well.  However, we do not know whether $\Sigma_M$ satisfies property (5).  The proof of the proposition above shows that there is a rank 3 valuated matroid on the generating set $\{ \varphi_e \mbox{ } \vert \mbox{ } e \in E \}$ such that every valuated circuit gives a tropical dependence, but we don't know if this extends to a valuated matroid on $\Sigma_M$ with this property.

This construction yields many examples of non-realizable tropical linear series.

\begin{lemma}
If $\Sigma_M$ is the tropicalization of a tropical linear series on a curve over a nonarchimedean field $K$, then $M$ is realizable over $K$.
\end{lemma}

\begin{proof}
If there is a rank 2 flat $f$ with $\vert f \vert = \vert E \vert - 1$, then $M$ is realizable over any infinite field.  We may therefore assume that $\vert E \vert \geq \vert f \vert +2$ for all rank 2 flats $f$.  We now prove that $M$ (with the trivial valuation) is the unique rank-3 valuated matroid on $E$ such that, if $V$ is a valuated circuit, then $\min \{ \varphi_e + V(\varphi_e) \}$ is a dependence.  By Proposition~\ref{Prop:Tropicalization} below, this implies that $M$ is realizable over $K$.

First, note that the circuits of $M$ are the minimal tropically dependent subsets of $\{ \varphi_e \mbox{ } \vert \mbox{ } e \in E \}$.  To see this, note that every set of 2 or fewer functions is tropically independent.  Further, if $\{ \varphi_{e_1}, \varphi_{e_2} , \varphi_{e_3} \}$ is a set of size 3 that is not contained in a flat, let $f$ be the unique flat containing $e_1$ and $e_2$.  Then, on the edge from $e_1$ to $f$, the 3 functions have distinct slopes, and are therefore tropically independent.  Moreover, this shows that up to tropical scaling, $\vartheta = \min_{e \in C} \{ \varphi_e \}$ is the unique dependence among the functions $\{ \varphi_e \mbox{ } \vert \mbox{ } e \in C \}$ for any circuit $C$ in $M$.

Let $M'$ be a valuated matroid of rank 3 with the property that if $V$ is a valuated circuit, then $\min \{ \varphi_e + V(\varphi_e) \}$ is a dependence.  By the above, the circuits of $M$ are the minimal tropically dependent subsets, so if $V$ is a valuated circuit of $M'$, then $\supp (V)$ contains a circuit of $M$.  Moreover, if $\supp (V)$ is equal to a circuit in $M$, then $V$ must be constant on its support.  On the other hand, if $\supp (V)$ strictly contains a circuit $C$ of $M$, then since $M'$ has rank 3, we must have $\vert C \vert = 3$ and $\vert \supp (V) \vert = 4$.  Since $\vert C \vert = 3$, there is a unique flat $f$ such that $C \subseteq f$.  Let $e_1 , e_2 \in E \smallsetminus f$.  Since every set of size 4 contains the support of a valuated circuit in $M'$ and that support contains a circuit in $M$, we see that there exist valuated circuits $V_1 , V_2$ in $M'$ such that $\supp (V_i) = C \cup \{ e_i \}$.  Let $C = \{ e_3 , e_4 , e_5 \}$.   We have $\varphi_{e_i} (f) = 0$ for $i \leq 2$ and $\varphi_{e_i} (f) = 1$ for $i \geq 3$, so $V_i (\varphi_{e_i}) \geq \min_{j \geq 3} \{ V_i (\varphi_{e_j}) + 1 \}$.  For $j \geq 3$, the only two functions in $\supp(V_i)$ that have the same slope along the edge from $e_j$ to $f$ are the functions $\varphi_{e_k}$ for $e_k \in C \smallsetminus \{ e_j \}$.  It follows that $V_i (\varphi_{e_j}) = V_i (\varphi_{e_k})$ for all $j,k \geq 3$.  By the circuit axiom, there exists a valuated circuit $W$ such that $W(\varphi_{e_5}) = \infty$, $W(\varphi_{e_3}) = V_1(\varphi_{e_3})$, and $W \geq \min \{ V_1, V_2 \}$.  By construction, $\varphi_{e_3} + W(\varphi_{e_3})$ obtains the minimum at $f$, but it is the only function with positive slope on the edge from $e_4$ to $f$.  Thus, the tropical linear combination $\vartheta = \min_{i \leq 4} \{\varphi_{e_i} + W(\varphi_{e_i}) \}$ is not a dependence.  It follows that, for every valuated circuit $V$ in $M'$, $\supp (V)$ is equal to a circuit in $M$, and $V$ is constant on its support.  In other words, $M'$ is equal to $M$ with the trivial valuation.
\end{proof}
\end{example}

\section{Tropicalization of Linear Series}
\label{Sec:TLS}

We now show that the tropicalization of a linear series is a strong tropical linear series.

\begin{proposition}
\label{Prop:Tropicalization}
Let $C$ be a curve over a valued field $K$ with a skeleton $\Gamma \subset C^{\an}$.  Let $D$ be a divisor on $C$ and $V \subseteq \cL (D)$ a linear series of rank $r$.  Then
\[
\Sigma = \trop (V) := \{ \trop (f) \in \PL (\Gamma) \mbox{ } \vert \mbox{ } f \in V \smallsetminus \{ 0 \} \}
\]
is a strong tropical linear series of rank $r$.  Indeed, if $K$ has value group $G$, then $\vert \Sigma \vert$ is a $G$-definable subset of $\vert \Trop (D) \vert$, and the valuated matroid in property (5) can be chosen to be realizable over $K$.
\end{proposition}

\begin{proof}
Condition (1) of Definition~\ref{Def:TLS} is \cite[Proposition~6.5]{M23}.  For completeness, we include the proof.  Following \cite{Baker08}, let $E$ be an effective divisor of degree $r$ on $\Gamma$ and $E_C$ be an effective divisor on $C$ such that $\Trop (E_C) = E$.  Since $V$ has rank $r$, there exists a function $f \in V$ such that $\ddiv (f) + D \geq E_C$.  Hence, $\ddiv (\trop (f)) + \Trop (D) \geq E$.

Condition (2) holds because $V$ is a vector space of dimension $r+1$, so any $r+2$ elements are linearly dependent.  If a set of nonzero functions is linearly dependent, then the set of their tropicalizations is tropically dependent, and the result follows.

Condition (3) of Definition~\ref{Def:TLS} follows from induction on $r$, the case $r=0$ being vacuous. Let $f_0 , \ldots , f_r \in V$ be functions satisfying $s_{\zeta} (\trop (f_i)) = s_{\zeta} [i]$.  Because the functions $\trop (f_i)$ have distinct slopes along $\zeta$, they are tropically independent, and thus the set of functions $\{ f_0 , \ldots , f_r \}$ is linearly independent.  It follows that the set of functions $\{ f_0 , \ldots , f_i \}$ spans a linear subseries $W$ of rank $i$.  By induction, $\trop (W) \subseteq \{ \varphi \in \Sigma \mbox{ } \vert \mbox{ } s_{\zeta} (\varphi) \leq s_{\zeta} [i] \}$ is a tropical linear subseries of rank $i$.

To see condition (4), note that $\vert \Sigma \vert$ is the image of the analytification of the projectivization of $V$, under the continuous map to $\vert \Trop(D) \vert$.  Since this analytification is compact, it follows that $\vert \Sigma \vert$ is compact, and hence closed.  To see that $|\Sigma|$ is also definable, choose a basis for $V$, giving an isomorphism $V \cong K^n$.   
By elimination of quantifiers in the theory of algebraically closed valued fields of fixed characteristic and residue characteristic \cite{Robinson56}, the image of a semi-algebraic subset of $K^n$ under the valuation is the set of $G$-rational points of a $G$-definable set.  Thus, the image of $V \smallsetminus \{ 0 \}$ under tropicalization is a definable subset of $\vert \Trop(D) \vert$.

Finally, each $\varphi \in \Sigma$ that is rational over the value group lifts to a rational function $f \in V$.  These rational functions give rise to a realizable matroid in which $v : V \to \mathbb{R} \cup \{ \infty \}$ is a valuated circuit if there is a minimal linear dependence $\sum_{i=0}^s a_i \cdot f_i = 0$ with $\val (a_i) = v(f_i)$.  Then property (5) is satisfied, because each such linear dependence tropicalizes to a tropical dependence.

\end{proof}

\subsection{Restrictions of Tropical Linear Series}
\label{Sec:Restrictions}

An important property of tropical linear series is that they are preserved by restrictions to subgraphs.

\begin{lemma}
\label{Lem:Restriction}
Let $\Gamma$ be a metric graph, $D$ a divisor on $\Gamma$, and $\Sigma \subseteq R(D)$ a tropical linear series or rank $r$ on $\Gamma$.  If $\Gamma' \subseteq \Gamma$ is a metric subgraph, let $D'$ be the divisor on $\Gamma'$ given by
\[
D'(w) := D(w) - \min_{\varphi \in \Sigma} \bigg\{ \sum_{\zeta} s_{\zeta}(\varphi) \bigg\} \mbox{ for all } w \in \Gamma',
\]
where the sum is over all tangent vectors $\zeta$ from $w$ into the complement of $\Gamma'$.  Then the restriction
\[
\Sigma \vert_{\Gamma'} := \{ \varphi \vert_{\Gamma'} \mbox{ } \vert \mbox{ } \varphi \in \Sigma \} \subseteq R(D')
\]
is a tropical linear series of rank $r$ on $\Gamma'$.
\end{lemma}

\begin{proof}
To see condition (1), let $E$ be an effective divisor of degree $r$ on $\Gamma'$.  Since $\Sigma$ is a tropical linear series of rank $r$ on $\Gamma$, there exists a function $\varphi \in \Sigma$ such that $\ddiv (\varphi) + D \geq  E$.  For any $w \in \Gamma'$, we have
\[
\ord_w (\varphi \vert_{\Gamma'}) = \ord_w (\varphi) + \sum_{\zeta} s_{\zeta} (\varphi),
\]
where the sum is over all tangent vectors $\zeta$ from $w$ into the complement of $\Gamma'$.  It follows that
\[
\ord_w (\varphi \vert_{\Gamma'})  \geq \ord_w (\varphi) + \min_{\psi \in \Sigma} \bigg \{ \sum_{\zeta} s_{\zeta}(\psi) \bigg \} \geq E(w) - D'(w),
\]
hence $\ddiv (\varphi \vert_{\Gamma'}) + D' \vert_{\Gamma'} \geq E$.

For condition (2), let $\{ \varphi_0 , \ldots , \varphi_{r+1} \}$ be a set of $r+2$ functions on $\Gamma$.  Since $\Sigma$ is a tropical linear series of rank $r$ on $\Gamma$, there exist coefficients $b_0 , \ldots , b_{r+1}$ such that $\min \{ \varphi_i + b_i \}$ achieves the minimum at least twice at every point of $\Gamma$.  But then $\min \{ \varphi_i \vert_{\Gamma'} + b_i \}$ achieves the minimum at least twice at every point of $\Gamma'$.

Condition (3) follows by induction on $r$.  The base case $r=0$ is trivial.  Let $\zeta$ be a tangent vector in $\Gamma'$ and $i<r$ a nonnegative integer.  By definition, the set
$\{ \varphi \in \Sigma \mbox{ } \vert \mbox{ } s_{\zeta} (\varphi) \leq s_{\zeta} [i] \}$
contains a tropical linear series $\Upsilon$ of rank $i$.  By induction, the restriction $\Upsilon \vert_{\Gamma'}$ is a tropical linear series of rank $i$ on $\Gamma'$, and the result follows.
\end{proof}

\begin{remark}
Lemma~\ref{Lem:Restriction} is important for applications.  If we restrict the tropicalization of a linear series to a connected subgraph $\Gamma' \subset \Gamma$ of genus $g'$, strictly smaller than the genus of $\Gamma$, then it is unclear whether this is the tropicalization of a linear series on a curve of genus $g'$ with skeleton $\Gamma'$.  However, since it still satisfies conditions (1)-(3) in Definition~\ref{Def:TLS}, we can still apply any constructions that depend only on these three properties.  This is essential in the proofs of \cite[Propositions~10.6 and 10.7]{M23} and \cite[Propositions~5.6 and 5.7]{M13}.
\end{remark}

The restriction property also holds for strong tropical linear series.

\begin{lemma}
\label{Lem:StrongRestriction}
Let $\Gamma$ be a metric graph and $\Sigma$ a strong tropical linear series or rank $r$ on $\Gamma$.  If $\Gamma' \subseteq \Gamma$ is a metric subgraph, then the restriction $\Sigma \vert_{\Gamma'}$ is a strong tropical linear series of rank $r$ on $\Gamma'$.
\end{lemma}

\begin{proof}
Condition (4) holds because the intersection of a closed subset with $\Sym^{d'} (\Gamma') \subseteq \Sym^d (\Gamma)$ is a closed subset of $\Sym^{d'} (\Gamma')$ and the intersection of definable subsets is definable.  Condition (5) holds because the restriction of a dependence on $\Gamma$ to $\Gamma'$ is a dependence on $\Gamma'$.
\end{proof}

\subsection{Dimension of tropical linear series} 
\label{Sec:Dim}

As explained in Example~\ref{Ex:Lollipop} above, most graphs $\Gamma$ have divisors $D$ such that the polyhedral set $|R(D)|$, parametrizing all effective divisors equivalent to $D$, has dimension strictly greater than the rank of $D$.  In such cases $R(D)$ is not the tropicalization of a linear series.  Here we discuss partial results regarding the dimensions of tropical linear series.

\begin{lemma}
\label{Lem:LowerDim}
If $\Sigma \subseteq R(D)$ is a tropical linear series of rank $r$, then $\dim (\vert \Sigma \vert) \geq r$.\end{lemma}

\begin{proof}
We exhibit a divisor $D' \in \vert \Sigma \vert$ such that the local dimension of $\vert \Sigma \vert$ in a neighborhood of $D'$ is at least $r$.  Let $\zeta$ be a tangent vector and let $\varphi_i \in \Sigma$ satisfy $s_{\zeta} (\varphi_i) = s_{\zeta} [i]$ for $0 \leq i \leq r$.  Let $I \subseteq \Gamma$ be a half-open interval containing $\zeta$ on which all of the functions $\varphi_i$ have constant slope, and let $\vartheta = \min \{ b_i + \varphi_i \}$ be a tropical linear combination with the property that each function $\varphi_i$ obtains the minimum on an open subinterval of $I$.  Letting $U \subset \mathbb{R}^{r+1}$ be a sufficiently small open ball around the vector $(b_0 , \ldots , b_r)$, we see that the map from $\Phi : U/(1,\ldots,1) \to \vert \Sigma \vert$ given by 
\[
\Phi (a_0 , \ldots , a_r ) = D + \ddiv \left( \min \{ a_0 + \varphi_0, \cdots, a_r + \varphi_r \} \right)
\]
is smooth and injective.  It follows that $\vert \Sigma \vert$ has dimension at least $r$ in a neighborhood of $D' = D + \ddiv (\vartheta)$.
\end{proof}

In order to prove a partial converse, we first establish a property of generating sets of tropical linear series.

\begin{lemma} \label{Lem:Submodules}
Let $\Sigma$ be a tropical linear series of rank $r$, and let $S \subseteq \Sigma$ be a generating set.  For every function $\varphi \in \Sigma$, there exists a subset $T \subseteq S$ of size $\vert T \vert \leq r+1$ such that $\varphi$ is contained in the submodule generated by $T$.
\end{lemma}

\begin{proof}
Let $\varphi \in \Sigma$ and let $T \subseteq S$ be minimal such that $\varphi$ is contained in the submodule generated by $T$.  Then there exist coefficients $a_{\psi} \in \mathbb{R}$ such that $\varphi = \min \{ a_{\psi} + \psi \vert \psi \in T \}$.  Because $T$ is minimal, for every $\psi \in T$, there is a point $v \in \Gamma$ such that $\psi$ achieves the minimum uniquely at $v$.  In other words, $\min \{ a_{\psi} + \psi \vert \psi \in T \}$ is a certificate of independence for $T$.  Since every set of $r+2$ functions in $\Sigma$ is tropically dependent, it follows that $\vert T \vert \leq r+1$.
\end{proof}

\begin{corollary}
\label{Cor:UpperDim}
Let $\Sigma \subset R(D)$ be a tropical linear series of rank $r$.  If $\Sigma$ is finitely generated, then $\mathrm{dim} \vert \Sigma \vert \leq r$.
\end{corollary}

\begin{proof}
By Lemma~\ref{Lem:Submodules},
\[
\Sigma = \bigcup_{T \subseteq S, \vert T \vert = r+1} \langle T \rangle .
\]
Let $T = \{ \varphi_0 , \ldots , \varphi_r \} \subseteq S$ be a subset of size $\vert T \vert = r+1$, and consider the map $\Phi : \mathbb{R}^{r+1} \to \vert \langle T \rangle \vert$ given by
\[
\Phi (a_0 , \ldots , a_r ) = D + \ddiv \left( \min \{ a_0 + \varphi_0, \cdots, a_r + \varphi_r \} \right) .
\]
Note that $\Phi$ is constant along the diagonal $a_0 = \cdots = a_r$, so $\Phi$ factors through $\mathbb{R}^{r+1} / (1,\ldots, 1)$. Because $\Gamma$ is compact, there is a positive integer $M$ such that $M + \varphi_i > \min_{j \neq i} \{ \varphi_j \}$ for all $i$.  It follows that the restriction of $\Phi$ to the cube $[0,M]^{r+1}$ is surjective.  This is a piecewise linear map, so $\mathrm{dim} \vert \langle T \rangle \vert \leq r$.  It follows that $\vert \Sigma \vert$ is a union of finitely many sets of dimension at most $r$, hence it has dimension at most $r$ as well.
\end{proof}

Note that Lemma~\ref{Lem:LowerDim} and Corollary~\ref{Cor:UpperDim} concern the \emph{global} dimension of $\vert \Sigma \vert$.  Even when $\Sigma$ is finitely generated, we do not yet know whether $\vert \Sigma \vert$ must be equidimensional.  It is a priori possible for there to exist a divisor $D' \in \vert \Sigma \vert$ such that the local dimension of $\vert \Sigma \vert$ in a neighborhood of $D'$ is smaller than $r$, but we do not know of any such examples.

\section{Rank One Tropical Linear Series}
\label{Sec:FG}

\subsection{Finite Generation}

In this section, we prove Theorem~\ref{Thm:Rank1} and show that every tropical linear series of rank 1 is a strong tropical linear series.

Given a tropical linear series $\Sigma$ on a graph $\Gamma$, we begin by defining a subdivision of $\Gamma$ for which the slope vectors are constant on edges of the subdivision.

\begin{lemma}
\label{Lem:Subdivision}
Given a tropical linear series $\Sigma \subseteq R(D)$ on a metric graph $\Gamma$, there exists a finite set $V \subset \Gamma$ such that:
\begin{enumerate}
\item  $V$ contains $\mathrm{Supp} (D)$ and all points of valence different from 2, and
\item  the slope vector $s_{\zeta} (\Sigma)$ is constant on each (oriented) edge of $\Gamma \smallsetminus V$.
\end{enumerate}
\end{lemma}

\begin{proof}
For each tangent vector $\zeta$ on $\Gamma$, let $\varphi^i_{\zeta} \in \Sigma$ be a function with $s_{\zeta} ( \varphi^i_{\zeta} ) = s_{\zeta} [i]$.  Let $I_{\zeta}$ be the half-open interval on which the functions $\varphi^i_{\zeta}$ have constant slope $s_{\zeta} [i]$ for all $i = 0, \ldots , r$.  For each point $v \in \Gamma$, let $U_v = \cup I_{\zeta}$,
where the union is over all tangent vectors $\zeta$ based at $v$.  Note that $U_v$ is an open set containing $v$.  Since the sets $U_v$ cover the compact space $\Gamma$, there exists a finite set $V$ such that $\{ U_v \vert v \in V \}$ is a cover.  The set $V$ therefore satisfies the second condition above.  Since $\mathrm{Supp} (D)$ is finite and there are only finitely many points of valence other than 2, the result follows.
\end{proof}

Lemma~\ref{Lem:Subdivision} allows us to construct, for each edge in the model, a set of functions in the tropical linear series $\Sigma$ that have specified behavior on that edge.

\begin{lemma}
\label{Lem:GeneratingSet}
Let $\Sigma \subseteq R(D)$ be a linear series on a metric graph $\Gamma$, let $V \subset \Gamma$ be the finite set as in Lemma~\ref{Lem:Subdivision}, and let $G$ be the model for $\Gamma$ induced by the vertex set $V$.  For every (oriented) edge $E$ in $G$ and every nonnegative integer $i \leq r$, there exists a function $\varphi^E_i \in \Sigma$ with $s_{\zeta} (\varphi^E_i) = s_{\zeta}[i]$ for all tangent vectors $\zeta$ in $E$.
\end{lemma}

\begin{proof}
Let $\zeta$ and $\zeta'$ denote the tangent vectors at the head and tail, respectively, of the edge $E$.  By definition, the set
$\{ \varphi \in \Sigma \mbox{ } \vert \mbox{ } s_{\zeta} (\varphi) \leq s_{\zeta} [i] \}$
contains a tropical linear series $\Upsilon$ of rank $i$.  By Lemma~\ref{Lem:Slopes}, there are exactly $i+1$ slopes $s_{\zeta'} (\varphi)$, as $\varphi$ ranges over $\Upsilon$.  These $i+1$ slopes are a subset of the $r+1$ slopes $s_{\zeta'} (\varphi)$, as $\varphi$ ranges over $\Sigma$.  It follows that at least one of these slopes must be greater than or equal to $s_{\zeta'} [i]$.  Let $\varphi^E_i \in \Upsilon$ be a function with
\[
s_{\zeta'} (\varphi^E_i) \geq s_{\zeta'} [i] = s_{\zeta} [i] .
\]
Since $E$ does not intersect $\mathrm{Supp}(D)$, the slope of $\varphi^E_i$ cannot increase on this interval.  Since $\varphi^E_i \in \Upsilon$, we have $s_{\zeta} (\varphi^E_i) \leq s_{\zeta} [i]$, and it follows that $s_{\eta} (\varphi^E_i) = s_{\zeta} [i]$ for all tangent vectors $\eta$ in $E$.
\end{proof}

Note that Lemmas~\ref{Lem:Subdivision} and~\ref{Lem:GeneratingSet} hold for tropical linear series of arbitrary rank.  
When the tropical linear series has rank 1, we have the following.

\begin{proposition}
\label{Prop:FiniteVertexSet}
Let $\Sigma \subseteq R(D)$ be a tropical linear series of rank 1.  Then there exists a finite set $W \subseteq \Gamma$ such that, for all $v \notin W$, there is a unique divisor $D' \in \vert \Sigma \vert$ such that $v \in \mathrm{Supp} (D')$.  Moreover, if $\ddiv (\psi) + D = D'$ is this unique divisor, and $v$ is contained in the edge $E$ of the model $G$, then $\psi \in \langle \varphi^E_0, \varphi^E_1 \rangle$.
\end{proposition}

\begin{proof}
Let $V \subset \Gamma$ be the finite set as in Lemma~\ref{Lem:Subdivision}, and let $E$ be an edge of the induced model $G$.  Let $W_E$ be the set of points $v \in E$ such that the functions $\varphi^E_0 + \varphi^E_1 (v)$ and $\varphi^E_1 + \varphi^E_0 (v)$ agree on an open subset of $\Gamma$.  Since $\varphi^E_0$ and $\varphi^E_1$ have different slopes along $E$ and have finitely many domains of linearity, $W_E$ is a finite set.  We define
\[
W = \bigcup_E W_E \cup V .
\]

Let $v \in \Gamma \smallsetminus V$.  Then $v$ is contained in an edge $E$ of the model $G$.  Let $\psi \in \Sigma$ be a function such that $\ddiv (\psi) + D \geq v$.  By definition, any 3 functions in $\Sigma$ are tropically dependent.  Thus, the functions $\psi, \varphi^E_0,$ and $\varphi^E_1$ are tropically dependent.  In other words, there exist coefficients $b, b_0, b_1$ such that
\[
\min \{ \psi + b, \varphi^E_0 + b_0 , \varphi^E_1 + b_1 \}
\]
occurs at least twice at every point of $\Gamma$.  By simultaneous tropical scaling, we may assume that $b=0$.  Since $\varphi^E_0$ and $\varphi^E_1$ have different slopes on $E$, we see that all 3 functions must simultaneously obtain the minimum at $v$.  This determines the coefficients $b_0$ and $b_1$, and in particular implies that $b_1 - b_0 = \varphi^E_0 (v) - \varphi^E_1 (v)$.

If $\psi \neq \min \{ \varphi^E_0 + b_0 , \varphi^E_1 + b_1 \}$, then in the tropical dependence between the three functions, there must be an open set where $\varphi^E_0 + b_0$ and $\varphi^E_1 + b_1$ simultaneously achieve the minimum.  It follows that $v \in W_E$.  In other words, if $v \notin W$, then $\psi$ is a unique tropical linear combination of $\varphi^E_0$ and $\varphi^E_1$.
\end{proof}

We now prove that, when $\Sigma$ has rank 1, the finite set of functions constructed in Lemma~\ref{Lem:GeneratingSet} generates $\Sigma$.

\begin{theorem}
\label{Thm:Rank1FG}
If $\Sigma \subseteq R(D)$ is a tropical linear series of rank 1, then it is generated by the functions $\varphi^E_i$.  More precisely, for each function $\psi \in \Sigma$, there exists an edge $E$ of $G$ such that $\psi \in \langle \varphi^E_0 , \varphi^E_1 \rangle$.  In particular, every tropical linear series of rank 1 is finitely generated.
\end{theorem}

\begin{proof}
For each edge $E$ of the model $G$, let $\Upsilon_E = \langle \varphi^E_0 , \varphi^E_1 \rangle \subseteq \Sigma$, and let $\Upsilon = \cup \Upsilon_E$.  Now, let $\psi \in \Sigma$ and $D' = \ddiv (\psi) + D$.  By Proposition~\ref{Prop:FiniteVertexSet}, if $\psi \notin \Upsilon$, then $\mathrm{Supp} (D') \subseteq W$.  Since $W$ is finite, it follows that $\vert \Sigma \vert \smallsetminus \vert \Upsilon \vert$ is finite.  Finite sets are closed, and by Lemma~\ref{Lem:FGClosed}, $\vert \Upsilon \vert$ is closed.  By Lemma~\ref{Lem:Connect}, however, $\vert \Sigma \vert$ is connected, hence $\vert \Sigma \vert \smallsetminus \vert \Upsilon \vert = \emptyset$.
\end{proof}

\subsection{Strong Tropical Linear Series of Rank 1}

Before proving that every tropical linear series of rank 1 is strong, we first prove the following useful lemma.

\begin{lemma}
\label{Lem:Ind3}
Let $x,y \in \Gamma$, and let $A = \{ \varphi_1 , \varphi_2 , \varphi_3 \}$ be a tropically dependent set.  If
\begin{enumerate}
\item $\varphi_1 (x) = \varphi_2 (x) < \varphi_3 (x)$, and
\item $\varphi_1 (y) = \varphi_3 (y) \leq \varphi_2 (y)$,
\end{enumerate}
then $\vartheta = \min_{\varphi_i \in A} \{ \varphi_i \}$ is a dependence.  
\end{lemma}

\begin{proof}
If there is a point $z \in \Gamma$ where $\varphi_1$ uniquely obtains the minimum in $\vartheta$, then, for sufficiently small $\epsilon > 0$, consider the tropical linear combination
\[
\vartheta' = \min \{ \varphi_1 , \varphi_2 - \epsilon , \varphi_3 - 2 \epsilon \} .
\]
We see that $\varphi_2 - \epsilon$ obtains the minimum uniquely at $x$, $\varphi_3 - 2 \epsilon$ obtains the minimum uniquely at $y$, and $\varphi_1$ obtains the minimum uniquely at $z$.  Thus, $\vartheta'$ is a certificate of independence, which is impossible, because $A$ is tropically dependent. 

Similarly, if there is a point in $\Gamma$ where $\varphi_2$ uniquely obtains the minimum, then
\[
\min \{ \varphi_1 - \epsilon , \varphi_2 , \varphi_3 - 2 \epsilon \}
\]
is a certificate of independence, and if there is a point in $\Gamma$ where $\varphi_3$ uniquely obtains the minimum, then
\[
\min \{ \varphi_1 - \epsilon , \varphi_2 - 2 \epsilon , \varphi_3 \}
\]
is a certificate of independence.  It follows that none of the 3 functions $\varphi_i$ obtains the minimum uniquely at any point, and thus $\vartheta$ is a dependence.
\end{proof}

\begin{lemma}
\label{Lem:Prop6}
Every tropical linear series $\Sigma$ of rank 1 satisfies condition (5) of the definition of strong tropical linear series.
\end{lemma}

\begin{proof}
Let $\mathcal{V}$ be the set of functions $V : \Sigma \to \mathbb{R} \cup \{ \infty \}$ such that $\vert \supp (V) \vert \leq 3$ and $\min \{ \varphi + V(\varphi) \}$ is a dependence.  We show that $\mathcal{V}$ is the set of valuated circuits of a valuated matroid on $\Sigma$.

Let $V_1 , V_2 \in \mathcal{V}$ and $\varphi_0 , \varphi_1 \in \Sigma$ with $V_1 (\varphi_0) = V_2 (\varphi_0) \neq \infty$ and $V_1 (\varphi_1) < V_2 (\varphi_1)$.  For simplicity, assume that no two elements of $A = \supp (V_1) \cup \supp (V_2)$ differ by a constant  We show that there exists $V \in \mathcal{V}$ such that $V (\varphi_0) = \infty$, $V( \varphi_1) = V_1 (\varphi_1)$, and $V \geq \min \{V_1 , V_2 \}$.

To construct $V$, first consider the tropical linear combination 
\[
\vartheta = \min_{\varphi \in A \smallsetminus \{ \varphi_0 , \varphi_1 \}} \{ \varphi + \min \{ V_1 (\varphi), V_2 (\varphi) \} \} .
\]
We first show that $V_1 (\varphi_1) + \varphi_1 \geq \vartheta$.  To see this, let $x \in \Gamma$.  Since $\min \{ \varphi + V_1 (\varphi) \}$ is a dependence, there exists $\varphi_i \in \supp(V_1)$, $i \neq 1$, such that $V_1 (\varphi_i) + \varphi_i (x) \leq V_1 (\varphi_1) + \varphi_1 (x)$.  If $i \neq 0$, then 
\[
\vartheta (x) \leq V_1 (\varphi_i) + \varphi_i (x) \leq V_1 (\varphi_1) + \varphi_1 (x).
\]
If $i=0$, then since $\min \{ \varphi + V_2 (\varphi) \}$ is a dependence and $V_1 (\varphi_0) = V_2 (\varphi_0)$, there exists $\varphi_j \in \supp (V_2)$ such that $V_2 (\varphi_j) + \varphi_j (x) \leq V_1 (\varphi_0) + \varphi_0 (x)$.  Note that $\varphi_j \neq \varphi_1$ because $V_1 (\varphi_1) < V_2 (\varphi_1)$.  It follows that 
\[
\vartheta (x) \leq V_2 (\varphi_j) + \varphi_j (x) \leq V_1 (\varphi_1) + \varphi_1 (x).
\]
As a consequence, we see that $c = \min \{ \varphi_1 (x) - \vartheta (x) \colon x \in \Gamma \}$ is nonnegative.
  
Now, let $c_i = c + \min \{ V_1 (\varphi_i) , V_2 (\varphi_i) \}$, and consider the tropical linear combination
\[
\vartheta' = \min_{\varphi_i \in A \smallsetminus \{ \varphi_0 , \varphi_1 \} } \{ c_i + \varphi_i \} . 
\]
Note that $V_1 (\varphi_1) + \varphi_1 \geq \vartheta'$, with equality at some point $x \in \Gamma$.  Let $\varphi_2 \in A \smallsetminus \{ \varphi_0 , \varphi_1 \}$ satisfy $V_1 (\varphi_1) + \varphi_1 (x) = c_2 + \varphi_2 (x)$.  We now show that there exists a point $y \in \Gamma$ and a function $\varphi_3 \in A \smallsetminus \{ \varphi_0 , \varphi_1 , \varphi_2 \}$ such that $c_3 + \varphi_3 (y) = \vartheta' (y)$.  Indeed, if $c_2 + \varphi_2 \neq \vartheta'$, then there exists a point $y \neq x$ such that $c_2 + \varphi_2 (y) > \vartheta' (y)$, and we choose $\varphi_3$ so that $c_3 + \varphi_3 (y) = \vartheta' (y)$.  Otherwise, we have $c_2 + \varphi_2 = \vartheta'$, and by the same argument that shows that $V_1 (\varphi_1) + \varphi_1 \geq \vartheta$, we see that $\varphi_2 \in \supp(V_1) \cap \supp(V_2)$ and $V_1 (\varphi_2) = V_2 (\varphi_2)$.  In this case, let $\varphi_3$ be the unique function in $\supp (V_2) \smallsetminus \{ \varphi_0 , \varphi_ 2 \}$.  Since $\min \{ \varphi + V_2 (\varphi) \}$ is a dependence, there exists a point $y \in \Gamma$ with $c_3 + \varphi_3 (y) = \vartheta' (y)$.  By slightly increasing the coefficient of $\varphi_2$ if necessary, we may assume that $y \neq x$.  Because $\varphi_3$ obtains the minimum at $y$, we see that $d = \min \{ V_1 (\varphi_1) + \varphi_1 (y) , c_2 + \varphi_2 (y) \} - (c_3 + \varphi_3 (y))$ is nonnegative.  

Finally, let $d_3 = c_3 + d$ and consider the function
\begin{displaymath}
V (\varphi_i) = \left\{ \begin{array}{ll}
V_1 (\varphi_1) & \textrm{if $i=1$,} \\
c_2 & \textrm{if $i=2$,}\\
d_3 & \textrm{if $i=3$,}\\
\infty & \textrm{otherwise.}
\end{array} \right.
\end{displaymath}
The functions $V_1 (\varphi_1) + \varphi_1$ and $c_2 + \varphi_2$ obtain the minimum at $x$, and the function $d_3 + \varphi_3$ obtains the minimum at $y$, along with one of the other two functions.  If $d$ is nonzero, then $d_3 + \varphi_3$ cannot obtain the minimum at $x$, and if $d=0$, then by construction, one of the other two functions does not obtain the minimum at $y$.  In other words, at least 2 of the 3 functions obtain the minimum at one of the 2 points, and exactly 2 of the 3 functions obtain the minimum at the other point.  Since any set of 3 functions in $\Sigma$ is dependent, it follows from Lemma~\ref{Lem:Ind3} that $\min \{ \varphi + V(\varphi) \}$ is a dependence.  Note that $V (\varphi_0) = \infty$, $V( \varphi_1) = V_1 (\varphi_1)$, and $V \geq \min \{V_1 , V_2 \}$.  Thus, $\mathcal{V}$ is the set of valuated circuits of a valuated matroid on $\Sigma$.
\end{proof}

\begin{corollary}
\label{Cor:WeakIsStrong}
Every tropical linear series $\Sigma$ of rank 1 is strong.  Moreover, $\vert \Sigma \vert$ is compact of pure dimension 1.
\end{corollary}

\begin{proof}
Since $\Sigma$ is finitely generated, $\vert \Sigma \vert$ is closed and definable by Lemma~\ref{Lem:FGClosed}, and it satisfies property (5) by Lemma~\ref{Lem:Prop6}.  It has dimension 1 by Corollary~\ref{Cor:UpperDim} and has no isolated points by Lemma~\ref{Lem:Connect}, and is therefore equidimensional.
\end{proof}

\subsection{Harmonic Morphisms}

Let $\Sigma = \langle \varphi_0 , \ldots , \varphi_n \rangle$ be a finitely generated strong tropical linear series.  By definition, there exists a valuated matroid $M$ on $\{ \varphi_0, \ldots \varphi_n \}$ such that, if $V$ is a valuated circuit of $M$, then 
$\min \{ \varphi_i + V(\varphi_i) \}$ is a tropical dependence.  The image of $\Gamma$ under the map $\Phi : \Gamma \to \mathbb{T}\mathbb{P}^n$ given by $\Phi = (\varphi_0 , \ldots , \varphi_n )$ is contained in
\[
\mathcal{B}(M) := \{ x \in \mathbb{T}\mathbb{P}^n \mbox{ } \vert \mbox { } \min_{i} \{V(i) + x_i \} \text{ occurs at least twice for all valuated circuits } V \} .
\]

\begin{proposition}
\label{Prop:Harmonic}
There exists a tropical modification $\widetilde{\Gamma}$ of $\Gamma$ and a balanced map $\widetilde{\Phi} : \widetilde{\Gamma} \to \mathcal{B}(M)$ such that $\widetilde{\Phi} \vert_{\Gamma} = \Phi$.
\end{proposition}

\begin{proof}
We construct the tropical modification $\widetilde{\Gamma}$ by adding infinite rays to $\Gamma$.  Specifically, for each $i$ and each point $x \in \Gamma$ with $D + \ddiv ( \varphi_i ) \geq x$, add an infinite edge to $\Gamma$ based at $x$.  We extend $\varphi_i$ to $\widetilde{\Gamma}$ as a linear function with slope $-\ord_x (\varphi_i)$ along the infinite edge based at $x$.  The resulting map $\widetilde{\Phi}$ given coordinatewise by the extended $\varphi_i$'s is balanced.
\end{proof}

In the rank 1 case, Proposition~\ref{Prop:Harmonic} has the following consequence.

\begin{theorem}
\label{Thm:Lifting}
Let $\Sigma$ be a tropical linear series of rank $1$ on $\Gamma$ and let $K$ be an algebraically closed and  nontrivially valued field of residue characteristic zero.  Then there is a curve $X$ over $K$ whose skeleton has underlying metric graph $\Gamma$ (possibly with vertices of positive genus) and a linear series of rank $1$ on $X$ whose tropicalization is $\Sigma$.
\end{theorem}

\begin{proof}
By Theorem~\ref{Thm:Rank1FG}, $\Sigma$ is finitely generated, and by Lemma~\ref{Lem:UniqueGen} there is a unique minimal generating set up to addition of scalars.  Let $\{ \varphi_0 , \ldots , \varphi_n \}$ be this generating set.  By Corollary~\ref{Cor:WeakIsStrong}, this set determines a rank 2 valuated matroid $M$.  The tropical variety $\mathcal{B}(M)$ is a metric tree.  By Proposition~\ref{Prop:Harmonic}, there exists a tropical modification $\widetilde{\Gamma}$ of $\Gamma$ and a balanced map $\widetilde{\Phi} : \widetilde{\Gamma} \to \mathcal{B}(M)$ such that $\widetilde{\Phi} \vert_{\Gamma} = \Phi$.  The map $\widetilde{\Phi}$ is finite because $d_{\zeta} (\widetilde{\Phi}) = s_{\zeta} [1] - s_{\zeta} [0] > 0$ for all tangent vectors $\zeta$.  The result then follows from \cite{ABBR2}.
\end{proof}

\begin{example}
We return to Example~\ref{Ex:Interval}, and consider the harmonic morphism $\widetilde{\Phi}$ constructed in Proposition~\ref{Prop:Harmonic}.  Recall that $\Gamma$ is an interval with left endpoint $x$ and right endpoint $y$, $D=2x$, and $\Sigma \subset R(D)$ is a rank 1 tropical linear series.  We let $\varphi_0 \in \Sigma$ satisfy $s_x (\varphi_0) = s_x [0]$ and $\varphi_1 \in \Sigma$ satisfy $s_y (\varphi_1) = s_y [1]$.  Recall that $\ddiv (\varphi_0) +D = x + w_0$ and $\ddiv (\varphi_1) +D = w_1 + y$.

In the case where $w_0$ is to the left of $w_1$, we have $\Sigma = \langle \varphi_0 , \varphi_1 \rangle$, and define $\widetilde{\Gamma}$ by attaching infinite legs to $\Gamma$ at $x, y, w_0,$ and $w_1$.  The map $\widetilde{\Gamma} \to \mathbb{TP}^1$ is depicted in Figure~\ref{Fig:EasyCase}.

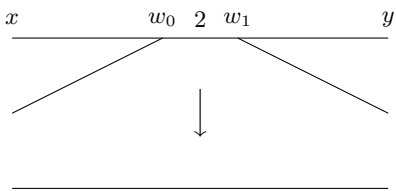
\begin{figure}[h!]
\begin{tikzpicture}

\draw (0,4)--(5,4);
\draw (0,3)--(2,4);
\draw (3,4) -- (5,3);
\draw (0,2) -- (5,2);

\draw (2.5,3) node {{$\Big\downarrow$}};

\draw (0,4.25) node {{\small $x$}};
\draw (5,4.25) node {{\small $y$}};
\draw (2,4.25) node {{\small $w_0$}};
\draw (3,4.25) node {{\small $w_1$}};
\draw (2.5,4.25) node {{\small $2$}};

\end{tikzpicture}
\caption{The harmonic morphism $\widetilde{\Phi}$ when $\Sigma = \langle \varphi_0 , \varphi_1 \rangle$.}
\label{Fig:EasyCase}
\end{figure}

In the case where $w_0$ is to the right of $w_1$, it is shown in Example~\ref{Ex:Interval} that there exists a function $\varphi_2 \in \Sigma$ that has slope 2 everywhere to the left of a point $z \in \Gamma$ and slope 0 everywhere to the right of $z$, and $\Sigma = \langle \varphi_0 , \varphi_1 , \varphi_2 \rangle$.  The tropical linear space $\mathcal{B} (M)$ is the standard tropical line in $\mathbb{TP}^2$:
\[
\mathcal{B}(M) = \{ (x_0 , x_1 , x_2 ) \in \mathbb{T}\mathbb{P}^2 \mbox{ } \vert \mbox { } \min \{ x_0 , x_1 , x_2 \} \text{ occurs at least twice } \} . 
\]
The map $\widetilde{\Gamma} \to \mathcal{B}(M)$ is depicted in Figure~\ref{Fig:HardCase}.

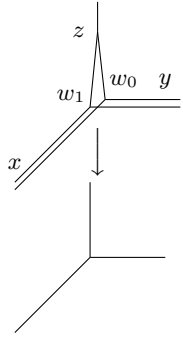
\begin{figure}[h!]
\begin{tikzpicture}

\draw (0,0)--(1,1);
\draw (1,1)--(1.1,2);
\draw (1.1,2) -- (1.2,1.1);
\draw (1.1,2) -- (1.1,2.4);
\draw (1.2,1.1) -- (2.2,1.1);
\draw (1,1)--(2.2,1);
\draw (1.2,1.1)--(0,-0.1);

\draw (1.1,0.4) node {{$\Big\downarrow$}};

\draw (0,-2)--(1,-1);
\draw (1,-1)--(1,0);
\draw (1,-1)--(2,-1);

\draw (0,0.25) node {{\small $x$}};
\draw (2,1.35) node {{\small $y$}};
\draw (1.45,1.35) node {{\small $w_0$}};
\draw (0.75,1.15) node {{\small $w_1$}};
\draw (0.85,2) node {{\small $z$}};

\end{tikzpicture}
\caption{The harmonic morphism $\widetilde{\Phi}$ when $\Sigma \neq \langle \varphi_0 , \varphi_1 \rangle$.}
\label{Fig:HardCase}
\end{figure}

\end{example}

\section{Open Questions}

\subsection{Questions}
\label{Sec:Qs}

We close the paper with some open questions about tropical linear series.  First, recall the questions about finite generation from the introduction.

\begin{question}[Question~\ref{Q:FG}]
Are all tropicalizations of linear series on algebraic curves finitely generated as tropical modules?
\end{question}

\begin{question}[Question~\ref{Q:FG-tropicalizations}]
Are all tropical linear series finitely generated as tropical modules?  If not, what about strong tropical linear series?
\end{question}

At present, we do not know which of the various topological and algebraic properties of tropical linear series imply the others.

\begin{question}
Let $\Sigma$ be a tropical linear series of rank $r$. 
\begin{enumerate}
\item  Is $\vert \Sigma \vert$ necessarily closed?  Is it definable?  Is it of equidimension $r$?
\item  If $\vert \Sigma \vert$ is closed, does it follow that $\vert \Sigma \vert$ is definable?  Equidimensional?
\item  What if we assume $\vert \Sigma \vert$ is both closed and definable? Or a strong tropical linear series?  Or the tropicalization of a linear series?
\item  Is every tropical linear series a strong tropical linear series?
\item Are there other implications among the properties in the definitions of tropical linear series and strong tropical linear series, e.g., do properties (1) and (5) together imply (3) and (4)?
\end{enumerate}
\end{question}

The realizability problem for tropical linear series is certainly interesting.  Theorem~\ref{Thm:Lifting} and Example~\ref{Ex:Matroid} represent early steps in this direction, and are highly suggestive.

\begin{question}
Aside from realizability of the associated valuated matroids, what other obstructions are there to realizing a strong tropical linear series as the tropicalization of a linear series?
\end{question}

With an eye toward the questions above, a few more technical questions could be helpful.  For example, recall that Lemma~\ref{Lem:GeneratorDependence} allows us to check property (2) by checking it on a generating set.

\begin{question}
Is there an analogue of Lemma~\ref{Lem:GeneratorDependence} for property (5), i.e. does property (5) hold for a tropical linear series if that analogous statement holds for a generating set?
\end{question}

\noindent As a first case, it would be interesting to determine whether or not (5) holds for $\Sigma_M$ in Example~\ref{Ex:Matroid}.

Similarly, it would be interesting to know if there is an analogue of Proposition~\ref{Prop:FiniteVertexSet} for tropical linear series of higher rank.  Such a result could be useful for proving that higher-rank tropical linear series are finitely generated.

\begin{question}
Let $\Sigma \subseteq R(D)$ be a tropical linear series of rank $r$.  Does there exist a finite set $W \subset \Gamma$ such that, for all divisors $E$ of degree $r$ whose support is disjoint from $W$, there is a unique $D' \in \vert \Sigma \vert$ such that $D' - E \geq 0$?
\end{question}

For a divisor $D$ on a metric graph $\Gamma$, define the \emph{TLS-rank} to be
\[
r_{TLS}(D) := \max \{ r \mbox{ } \vert \mbox{ } R(D) \text{ contains a tropical linear series of rank } r \}.
\]
\begin{question}
Does the TLS-rank satisfy Riemann-Roch?  In other words, does the following equality hold for all divisors $D$ on a metric graph of genus $g$:
\[
r_{TLS} (D) - r_{TLS} (K_{\Gamma} - D) = \deg (D) - g + 1 ?
\]
\end{question}

\noindent Given a metric graph $\Gamma$, one defines the set:
\[
W^r_d (\Gamma) = \{ D \in \Pic^d (\Gamma) \mbox{ } \vert \mbox{ } r(D) \geq r \} .
\]
We define:
\[
G^r_d (\Gamma) := \{ (D,\Sigma) \mbox{ } \vert \mbox{ } D \in \Pic^d (\Gamma), \Sigma \subseteq R(D) \text{ a tropical linear series of rank } r \}.
\]
The image of $G^r_d (\Gamma)$ in $\Pic^d (\Gamma)$ is
\[
\widetilde{W}^r_d (\Gamma) = \{ D \in \Pic^d (\Gamma) \mbox{ } \vert \mbox{ } r_{TLS} (D) \geq r \} .
\]
The set $\widetilde{W}^r_d (\Gamma)$ is contained in $W^r_d (\Gamma)$.  As shown in Examples~\ref{Ex:Harmonic} and~\ref{Ex:LoopOfLoops}, this containment is sometimes strict.

\begin{question}
Do the sets $G^r_d (\Gamma)$ and $\widetilde{W}^r_d (\Gamma)$ have the structure of tropical varieties?
\end{question}

Our interest in tropical linear series was partly motivated by the study of multiplication maps in \cite{M23, M13}.  To that end, we ask the following.

\begin{question}
Given tropical linear series $\Sigma_1 \subseteq R(D_1)$, $\Sigma_2 \subseteq R(D_2)$, define the multiplication map $\mu \colon \Sigma_1 \times \Sigma_2 \to R(D_1 + D_2)$ by $\mu (f_1 , f_2) = f_1 + f_2$.
\begin{enumerate}
\item  Is there necessarily a tropical linear series $\Sigma \subset R(D_1 + D_2)$  that contains the image of $\mu$?
\item  More generally, is there a procedure that determines whether a tropical module $\Sigma \subseteq R(D)$ is contained in a tropical linear series of a given rank?
\end{enumerate}
\end{question}

\begin{question}
Is there a natural converse to Proposition~\ref{Prop:Harmonic}, i.e., does every harmonic map from a tropical modification of $\Gamma$ to a tropical linear space come from a finitely generated strong tropical linear series?
\end{question}

\bibliography{TLS}

\begin{thebibliography}{CDPR12}

\bibitem[AB15]{AminiBaker15}
O.~Amini and M.~Baker.
\newblock Linear series on metrized complexes of algebraic curves.
\newblock {\em Math. Ann.}, 362(1-2):55--106, 2015.

\bibitem[ABBR15]{ABBR2}
O.~Amini, M.~Baker, E.~Brugall\'{e}, and J.~Rabinoff.
\newblock Lifting harmonic morphisms {II}: {T}ropical curves and metrized
  complexes.
\newblock {\em Algebra Number Theory}, 9(2):267--315, 2015.

\bibitem[AG22]{AminiGierczak22}
O.~Amini and L.~Gierczak.
\newblock Limit linear series: {C}ombinatorial theory.
\newblock preprint, 2022.

\bibitem[Ami14]{Amini14}
O.~Amini.
\newblock Equidistribution of {W}eierstrass points on curves over
  non-{A}rchimedean fields.
\newblock arXiv:1412.0926, 2014.

\bibitem[Bak08]{Baker08}
M.~Baker.
\newblock Specialization of linear systems from curves to graphs.
\newblock {\em Algebra Number Theory}, 2(6):613--653, 2008.

\bibitem[BN07]{BakerNorine07}
M.~Baker and S.~Norine.
\newblock {R}iemann-{R}och and {A}bel-{J}acobi theory on a finite graph.
\newblock {\em Adv. Math.}, 215(2):766--788, 2007.

\bibitem[BR15]{BakerRabinoff15}
M.~Baker and J.~Rabinoff.
\newblock The skeleton of the {J}acobian, the {J}acobian of the skeleton, and
  lifting meromorphic functions from tropical to algebraic curves.
\newblock {\em Int. Math. Res. Not. IMRN}, (16):7436--7472, 2015.

\bibitem[Car15]{Cartwright15}
D.~Cartwright.
\newblock Lifting matroid divisors on tropical curves.
\newblock {\em Res. Math. Sci.}, 2:Art. 23, 24, 2015.

\bibitem[CDPR12]{tropicalBN}
F.~Cools, J.~Draisma, S.~Payne, and E.~Robeva.
\newblock A tropical proof of the {B}rill-{N}oether theorem.
\newblock {\em Adv. Math.}, 230(2):759--776, 2012.

\bibitem[CJ20]{CPJ20}
K.~{Cook-Powell} and D.~Jensen.
\newblock Tropical methods in {H}urwitz-{B}rill-{N}oether theory.
\newblock Preprint, arXiv:2007.13877v1, 2020.

\bibitem[DS04]{DevelinSturmfels04}
M.~Develin and B.~Sturmfels.
\newblock Tropical convexity.
\newblock {\em Doc. Math.}, 9:1--27 (electronic), 2004.

\bibitem[FJP20]{M23}
G.~Fakras, D.~Jensen, and S.~Payne.
\newblock The {K}odaira dimensions of ${M}_{22}$ and ${M}_{23}$.
\newblock Preprint, arXiv:2005.00622v1, 2020.

\bibitem[FJP21]{M13}
G.~Farkas, D.~Jensen, and S.~Payne.
\newblock The non-abelian {B}rill-{N}oether divisor on
  {$\overline{\mathcal{M}}_{13}$} and the {K}odaira dimension of
  {$\overline{\mathcal{R}}_{13}$}.
\newblock Preprint, arXiv:2110.09553, 2021.

\bibitem[GK08]{GathmannKerber08}
A.~Gathmann and M.~Kerber.
\newblock A {R}iemann-{R}och theorem in tropical geometry.
\newblock {\em Math. Z.}, 259(1):217--230, 2008.

\bibitem[HMY12]{HMY12}
C.~Haase, G.~Musiker, and J.~Yu.
\newblock Linear systems on tropical curves.
\newblock {\em Math. Z.}, 270(3-4):1111--1140, 2012.

\bibitem[JP14]{tropicalGP}
D.~Jensen and S.~Payne.
\newblock Tropical independence {I}: {S}hapes of divisors and a proof of the
  {G}ieseker-{P}etri theorem.
\newblock {\em Algebra Number Theory}, 8(9):2043--2066, 2014.

\bibitem[JR21]{JensenRanganathan21}
D.~Jensen and D.~Ranganathan.
\newblock Brill-{N}oether theory for curves of a fixed gonality.
\newblock {\em Forum Math. Pi}, 9:Paper No. e1, 33, 2021.

\bibitem[Lar21]{Larson21}
H.~Larson.
\newblock A refined {B}rill-{N}oether theory over {H}urwitz spaces.
\newblock {\em Invent. Math.}, 224(3):767--790, 2021.

\bibitem[LLV20]{LLV}
E.~Larson, H.~Larson, and I.~Vogt.
\newblock Global {B}rill-{N}oether theory over the {H}urwitz space.
\newblock Preprint, arXiv:2008.10765, 2020.

\bibitem[LPP12]{LPP12}
C.~M. Lim, S.~Payne, and N.~Potashnik.
\newblock A note on {B}rill-{N}oether theory and rank-determining sets for
  metric graphs.
\newblock {\em Int. Math. Res. Not. IMRN}, (23):5484--5504, 2012.

\bibitem[Luo18]{Luo18}
Y.~Luo.
\newblock Idempotent analysis, tropical convexity, and reduced divisors.
\newblock Preprint, arXiv:1808.01987, 2018.

\bibitem[Mar67]{Martens67}
H.~Martens.
\newblock On the varieties of special divisors on a curve.
\newblock {\em J. Reine Angew. Math.}, 227:111--120, 1967.

\bibitem[Mar02]{Marker02}
D.~Marker.
\newblock {\em Model theory}, volume 217 of {\em Graduate Texts in
  Mathematics}.
\newblock Springer-Verlag, New York, 2002.
\newblock An introduction.

\bibitem[MS15]{MaclaganSturmfels}
D.~Maclagan and B.~Sturmfels.
\newblock {\em Introduction to tropical geometry}, volume 161 of {\em Graduate
  Studies in Mathematics}.
\newblock American Mathematical Society, Providence, RI, 2015.

\bibitem[MUW21]{MUW21}
M.~M\"{o}ller, M.~Ulirsch, and A.~Werner.
\newblock Realizability of tropical canonical divisors.
\newblock {\em J. Eur. Math. Soc. (JEMS)}, 23(1):185--217, 2021.

\bibitem[MZ08]{MikhalkinZharkov08}
G.~Mikhalkin and I.~Zharkov.
\newblock Tropical curves, their {J}acobians and theta functions.
\newblock In {\em Curves and abelian varieties}, volume 465 of {\em Contemp.
  Math.}, pages 203--230. Amer. Math. Soc., Providence, RI, 2008.

\bibitem[Pfl17]{Pflueger17b}
N.~Pflueger.
\newblock Brill-{N}oether varieties of $k$-gonal curves.
\newblock {\em Adv. Math.}, 312:46--63, 2017.

\bibitem[Rob56]{Robinson56}
A.~Robinson.
\newblock {\em Complete theories}.
\newblock North-Holland Publishing Co., Amsterdam, 1956.

\end{thebibliography}

\end{document}